\documentclass[10pt,reqno]{article}

\usepackage{amsmath}
\usepackage{latexsym}
\usepackage{amsmath,amsthm}
\usepackage{amsfonts}
\usepackage{amssymb}
\usepackage{amsfonts,euscript}
\usepackage{graphicx}
\usepackage[usenames,dvipsnames]{color}
\DeclareGraphicsExtensions{.eps,epsfig,.bmp,.jpg,.pdf,.mps,.png,.gif}
\newtheorem{theorem}{\bf Theorem}[section]

\newtheorem{lemma}{\bf Lemma}[section]

\newcommand{\beq}{\begin{equation}}
\newcommand{\eeq}{\end{equation}}
\newcommand{\beqn}{\begin{eqnarray}}
\newcommand{\eeqn}{\end{eqnarray}}
\newcommand{\bear}{\begin{array}}
\newcommand{\eear}{\end{array}}
\newcommand{\beit}{\begin{itemize}}
\newcommand{\eeit}{\end{itemize}}
\newcommand{\beqno}{\begin{eqnarray*}}
\newcommand{\eeqno}{\end{eqnarray*}}

\let\theta\vartheta
\def\eqref #1{(\ref{#1})}

\allowdisplaybreaks

\numberwithin{equation}{section}

\begin{document}

\title{Boundary stabilizing actuators for multi-phase fluids in a channel }
\author{}
\date{}
\maketitle
\begin{center}
\vskip-1cm
{\large\sc Ionu\c t Munteanu}\\
{\normalsize e-mail: {\tt ionut.munteanu@uaic.ro}}\\
[.25cm]

{\small Alexandru Ioan Cuza University, Department of Mathematics, }\\
{\small and Octav Mayer Institute of Mathematics (Romanian Academy)}\\
{\small Carol I no. 11, 700506 Ia\c{s}i, Romania}\\
[.25cm]

\end{center}

\begin{abstract}
This work is devoted to the problem of boundary stabilization of a mixture of two viscous and incompressible fluids in a three dimensional channel-like domain $(x,y,z)\in \mathbb{R}\times (0,1)\times\mathbb{R}$. The model consists of the Navier-Stokes equations, governing the fluid velocity, coupled with a convective Cahn-Hilliard equation for the relative density of the atoms of one of the fluids. The feedback controller, we design here, is with actuation  only on the tangential component of the velocity field on both walls $y=0,1$.  Moreover, it is given in a simple and explicit form, easy to manipulate from the computational point of view. The stability is guaranteed no matter how large the kinematic viscosity of the mixture is. Also it is independent of the other coefficients of the system.


\end{abstract}
\section{Introduction of the model and the main results}
In this work, we are concerned with the turbulence issues of binary mixture flows  in an infinite channel under shear. This situation arises, for example, in viscometric experiments. Numerous physical papers (see \cite{chella,doi}, for example) propose to consider, in this case, a coupling of Cahn-Hilliard and Navier-Stokes equations. More exactly, we shall  study the following coupling Cahn-Hilliard-Navier-Stokes system, also known as "model H",   cf. Hohenberg and Halperin \cite{t7} and Gurtin et al. \cite{t6}: 
\begin{flalign}\label{e1}\begin{array}{l}\left\{\begin{array}{l} u_t- \nu\Delta u+uu_x+vu_y+wu_z+p_x=\kappa \mu \varphi_x,\\
v_t- \nu\Delta v+uv_x+vv_y+wv_z+p_y=\kappa \mu \varphi_y,\\
w_t- \nu\Delta w+uw_x+vw_y+ww_z+p_z=\kappa \mu\varphi_z,\\
u_x+v_y+w_z=0,\\
\mu=-\varepsilon\Delta \varphi+\alpha F'(\varphi),\\
\varphi_t+u\varphi_x+v\varphi_y+w\varphi_z-\rho_0\Delta\mu=0,\\
t>0,\ x\in\mathbb{R},\ y\in (0,1),\ z\in\mathbb{R};\end{array}\right.\ \end{array}&&&&\end{flalign}supplemented with the boundary conditions
\begin{flalign}\label{e2}\begin{array}{l}\left\{\begin{array}{l}\text{ $u,v,w,p,\varphi$  $2\pi-$periodic in both $x$ and $z$ directions,}\\
u(t,x,0,z)=\mathbf{S}(t,x,z) \text{ and }u(t,x,1,z)=\mathbf{T}(t,x,z),\\
v(t,x,0,z)=v(t,x,1,z)=w(t,x,0,z)=w(t,x,1,z)=0,\\
\varphi_y(t,x,0,z)=\varphi_y(t,x,1,z)=(\Delta \varphi)_y(t,x,0,z)=(\Delta\varphi)_y(t,x,1,z)=0,\\
t>0,\ x,z\in\mathbb{R};\end{array}\right.\ \end{array}&&&&\end{flalign}and the initial data
\begin{flalign}\label{pe2}\left\{\begin{array}{l}u(0,x,y,z)=u_o(x,y,z),\ v(0,x,y,z)=v_o(x,y,z),\ w(0,x,y,z)=w_o(x,y,z),\\
\varphi(0,x,y,z)=\varphi_o(x,y,z),\ x\in\mathbb{R},\ y\in (0,1),\ z\in\mathbb{R}.\end{array}\right.\ &&&&\end{flalign}

Here $(u,v,w)$ is the velocity field (the tangential, normal and span-wise components, respectively), $p$ is the pressure, $\varphi$ represents the relative density of one species of atoms. $\nu>0$ is the kinematic viscosity of the mixture,  $\rho_0>0,\kappa>0$ are the mobility constant and capillarity (stress), respectively; $\varepsilon$ and $\alpha$ are two positive parameters describing the interactions between the two phases. In particular, $\varepsilon$ is related to the thickness of the interface separating the two fluids. Here $\mu$ is the chemical potential of the binary mixture and $F$ is the double-well potential  
$$F(\varphi)=\frac{1}{4}(\varphi^2-1)^2.$$ Finally, $\mathbf{S}$ and $\mathbf{T}$ are the controls, with actuation  on the tangential component of the velocity field, at the walls $y=0$ and $y=1$ only. The tangential actuation is technologically feasible because of the work on synthetic jets of Glezer \cite{G4}. In \cite{G4} it is shown that a teamed up pair of
synthetic jets can achieve an angle of $85^\circ$ from the normal direction, with the same
momentum as wall-normal actuation. The tangential velocity actuation is generated
using arrays of rotating disks. We emphasize that the  $2\pi-$ periodicity assumption is often assumed, for mathematical convenience, because it does not alter the essential features of the behaviour of the flow mixture (see for details \cite{a7}). (We notice that instead of the $2\pi-$periodicity we could have considered a general $L,l-$periodicity, $L,l>0$, i.e., elongated domains. Anyway, we confined ourselves to $2\pi-$ periodicity case for the ease of Fourier functional settings below, and since, easily seen,  the results obtained here hold true for the general periodic case as-well.)

The system \eqref{e1}-\eqref{pe2} describes the motion of two macroscopically immiscible, viscous, incompressible Newtonian
fluids. The model takes a partial mixing on a small length scale measured by a
parameter $\varepsilon> 0$ into account. Therefore the classical sharp interface between both
fluids is replaced by an interface region and an order parameter related to the
concentration difference of both fluids is introduced. Concerning the existence and uniqueness of solutions for such system we refer to the work \cite{boyer1}. The asymptotic behavior of \eqref{e1} has been previously investigated in many papers, see for instance \cite{boyer1,gal,gal1,gal2,gal3,gal4}, while, to the best of our knowledge, the problem of asymptotic exponential stabilization of such a system is for the first time addressed here.  In \cite[Theorem 2.3]{gal1} it is showed that the stationary solution of \eqref{e1} $(U,0,0,\text{const.})$, where $U$ is the Poiseuille profile below, is  unstable. Hence, even for the very simple case: constant target concentration $\varphi$, the problem of stabilization is not trivial. However, we shall not consider here the case of constant steady-state concentration because of two reasons: firstly, in such  case the problem mainly reduces to the problem of stabilization of the Navier-Stokes flows in a channel, problem which was already solved in e.g. \cite{barbutg, book, raymond, krstic} and many other papers; secondly, constant target concentration  means
either a completely mixed state or a pure state, while  a non-constant
state seems more physically relevant. Moreover,  since the set of stationary states is a
continuum, the convergence of a solution to a single
equilibrium is not a trivial issue. Concerning the stabilization problem associated to the Cahn-Hilliard system we refer to the results in \cite{barbucolli, book}

We notice that, the boundary conditions and the sixth equation in \eqref{e1} assure the mass conservation. Besides this, let us observe that, in the term $\kappa\mu\nabla\varphi$, the part $\alpha F'(\varphi)\nabla\varphi=\nabla (\alpha F(\varphi))$ can be viewed as a pressure term. In this light, here and in what follows we shall consider the equivalent of equation \eqref{e1}
\begin{flalign}\label{e4}\left\{\begin{array}{l} u_t- \nu\Delta u+uu_x+vu_y+wu_z+p_x=-\varepsilon\kappa \Delta\varphi \varphi_x,\\
v_t- \nu\Delta v+uv_x+vv_y+wv_z+p_y=-\varepsilon\kappa \Delta\varphi \varphi_y,\\
w_t- \nu\Delta w+uw_x+vw_y+ww_z+p_z=-\varepsilon\kappa \Delta\varphi\varphi_z,\\
u_x+v_y+w_z=0,\\
\mu=-\varepsilon\Delta \varphi+\alpha F'(\varphi),\\
\varphi_t+u\varphi_x+v\varphi_y+w\varphi_z-\rho_0\Delta\mu=0,\\
t>0,\ x\in\mathbb{R},\ y\in (0,1),\ z\in\mathbb{R}.\end{array}\right.\ &&\end{flalign}

  In order to compute the equilibrium state, in \eqref{e2} we put the boundary controls to be zero, i.e.,
	$$u(t,x,0,z)=u(t,x,1,z)=0,\ \forall x,z\in\mathbb{R}, \ t>0.$$ The stationary system reads as
	\begin{flalign}\label{joi1}\left\{\begin{array}{l} - \nu\Delta u+uu_x+vu_y+wu_z+p_x=-\varepsilon\kappa \Delta\varphi \varphi_x,\\
	- \nu\Delta v+uv_x+vv_y+wv_z+p_y=-\varepsilon\kappa \Delta\varphi \varphi_y,\\
	- \nu\Delta w+uw_x+vw_y+ww_z+p_z=-\varepsilon\kappa \Delta\varphi\varphi_z,\\
	u_x+v_y+w_z=0,\\
	\mu=-\varepsilon\Delta \varphi+\alpha F'(\varphi),\\
	u\varphi_x+v\varphi_y+w\varphi_z-\rho_0\Delta\mu=0,\\
	x\in\mathbb{R},\ y\in (0,1),\ z\in\mathbb{R},\end{array}\right.\ &&\end{flalign} and B.C.
	\begin{flalign}\begin{array}{l}\left\{\begin{array}{l}\text{ $u,v,w,p,\varphi$  $2\pi-$periodic in both $x$ and $z$ directions,}\\
	u(x,0,z)=0 \text{ and }u(x,1,z)=0,\\
	v(x,0,z)=v(x,1,z)=w(x,0,z)=w(x,1,z)=0,\\
	\varphi_y(x,0,z)=\varphi_y(x,1,z)=(\Delta \varphi)_y(x,0,z)=(\Delta\varphi)_y(x,1,z)=0,\\
	\ x,z\in\mathbb{R};\end{array}\right.\ \end{array}&&&&\end{flalign}In order to solve it, we consider a steady-state flow with zero wall-normal and span-wise velocity, i.e., $(U,0,0)$. Namely, we consider the classical Poiseuille profile
	\begin{flalign}\label{e10001}U(y)=-C_U(y^2-y),\ y\in(0,1) ,\end{flalign}for some $C_U>0$. The target concentration $\varphi_{tg}$, will be assumed as
	$$\varphi_{tg}=\varphi_{tg}(y).$$ This is  related to the fact that there is no turbulence in the stationary system, i.e., at each level $(x,z)$, the concentration is the same, which means that there is no vortex (no turbulence). Substituting this in the corresponding stationary sixth equation of \eqref{joi1}, we deduce that the target steady-state concentration $\varphi_{tg}$ satisfies
	$$\Delta\left[-\varepsilon\Delta\varphi_{tg}+\alpha (\varphi^3_{tg}-\varphi_{tg})\right]=0, \ y\in(0,1),$$
	$$(\Delta\varphi_{tg})_{y}(0)=(\Delta\varphi_{tg})_{y}(1)=0.$$Whence
	\begin{equation}\label{e130}\left\{\begin{array}{l}-\varepsilon\Delta\varphi_{tg}+\alpha(\varphi_{tg}^3-\varphi_{tg})=0 \text{ a.e. in }(0,1)\\
	(\varphi_{tg})_y(0)=(\varphi_{tg})_y(1)=0.\end{array}\right.\ \end{equation}Equation \eqref{e130}  has a solution $\varphi_{tg}\in H^1(0,1)$ which is, in fact, the minimizer of the l.s.c. and coercive functional
	$$\Upsilon(\varphi)=\int_0^1\left(\frac{\varepsilon}{2}|\nabla\varphi|^2+\alpha\frac{(\varphi^2-1)^2}{4}\right)dy.$$
	(For further details, see \cite[Lemma A1]{barbucolli}.) 
	
	For a function $f:[0,1]\rightarrow \mathbb{C}$, we denote by $\hat{f}:[0,1]\rightarrow\mathbb{C}$:
	$$\hat{f}(y):=f(1-y),\ y\in[0,1].$$We say that $f$ is symmetric if $f\equiv\hat{f}$, and antisymmetric if $f\equiv-\hat{f}$. We assume that, except some small regions, the target concentration is antisymmetric.  We set
	\begin{equation}\label{e159} \varphi_\infty(y):=\frac{1}{2}(\varphi_{tg}(y)-\varphi_{tg}(1-y)),\ y\in (0,1).\end{equation} $\varphi_\infty$ can be seen as the antisymmetric part of $\varphi_{tg}$, since $$\varphi_\infty(y)=-\varphi_\infty(1-y),\ y\in(0,1),$$and
	$$\varphi_{tg}=\varphi_\infty+\frac{\varphi_{tg}(y)+\varphi_{tg}(1-y)}{2}.$$
	
	Besides this, we shall assume that
			\begin{equation}\label{e135}(H_{0}) \ \ \ \ \ \ \ \  3\int_0^1\varphi_{tg}^2(y)dy-1\geq0. \ \ \ \ \ \ \end{equation}
	
	An example of $\varphi_{tg}$ is: $\varphi_{tg}(y)=-1,\ y\in(0,1/2-\varepsilon),\ \varphi_{tg}(y)=1,\ y\in (1/2+\varepsilon)$ and, on $(1/2-\varepsilon,1/2+\varepsilon)$, $\varphi_{tg}$ satisfies
	\begin{equation}\left\{\begin{array}{l}-\varepsilon\Delta\varphi_{tg}+\alpha(\varphi_{tg}^3-\varphi_{tg})=0 \text{ a.e. in }(1/2-\varepsilon,1/2+\varepsilon)\\
	(\varphi_{tg})_y(1/2-\varepsilon)=(\varphi_{tg})_y(1/2+\varepsilon)=0.\end{array}\right.\ \end{equation} Note that, in this case $(H_0)$ is full-filled.
	
	Next, defining the fluctuation variables
	$$u:=u-U,\ v:=v,\ w:=w,\ \varphi:=\varphi-\varphi_{tg},$$ we equivalently rewrite \eqref{e4} as
	\begin{flalign}\label{e6}\left\{\begin{array}{l}u_t- \nu\Delta u+(u+U)u_x+v(u+U)_y+wu_z+p_x=-\varepsilon\kappa[\Delta \varphi_{tg}\varphi_x+\Delta\varphi\varphi_x],\\
v_t- \nu\Delta v+(u+U)v_x+vv_y+wv_z+p_y=-\varepsilon\kappa[\Delta\varphi(\varphi_{tg})_y+\Delta \varphi_{tg}\varphi_y+\Delta\varphi\varphi_y],\\
w_t-\nu \Delta w+(u+U)w_x+vw_y+ww_z+p_z=-\varepsilon\kappa[\Delta \varphi_{tg}\varphi_z+\Delta\varphi\varphi_z],\\
u_x+v_y+w_z=0,\\
\begin{aligned}\varphi_t+&(u+U)\varphi_x+  v(\varphi+\varphi_{tg})_y+w\varphi_z\\&
-\rho_0\Delta\left\{-\varepsilon\Delta\varphi+\alpha [\varphi^3+3\varphi^2\varphi_{tg}+F''(\varphi_{tg})\varphi)]\right\}=0,\end{aligned}\\
t>0,\ x\in\mathbb{R},\ y\in(0,1),\ z\in\mathbb{R},\end{array}\right.\ &&\end{flalign}supplemented with the B.C. and initial data
\begin{flalign}\label{e7}\left\{\begin{array}{l}\text{ $u,v,w,p,\varphi$  $2\pi-$periodic in both $x$ and $z$ directions,}\\
u(t,x,0,z)=\mathbf{S}(t,x,z) \text{ and }u(t,x,1,z)=\mathbf{T}(t,x,z),\\
v(t,x,0,z)=v(t,x,1,z)=w(t,x,0,z)=w(t,x,1,z)=0,\\
\varphi_y(t,x,0,z)=\varphi_y(t,x,1,z)=(\Delta \varphi)_y(t,x,0,z)=(\Delta\varphi)_y(t,x,1,z)=0,\\
t>0,\ x,z\in\mathbb{R};\\
u(0)=u^o:=u_o-U,\ v(0)=v^o:=v_o,\ w(0)=w^o:=w_o,\\
 \varphi(0)=\varphi^o:=\varphi_o-\varphi_{tg},\ x\in\mathbb{R},\ y\in (0,1),\ z\in\mathbb{R}.\end{array}\right.\ &&\end{flalign}

In this work, we study only the stabilization problem associated to the linearized system of \eqref{e6}-\eqref{e7}. For computational reasons, the linearized we shall consider here will not be taken in the classical way. To this purpose, let us introduce the quantity 
\begin{flalign}\label{e8}\begin{aligned} \gamma:=\rho_0\alpha\int_0^1 F''(\varphi_{tg})dy,\end{aligned}\end{flalign}which by \eqref{e135}, we see that
\begin{equation}\label{e44}\gamma\geq0.\end{equation} 
  Hence, the linear system we study is the following
\begin{equation}\label{e9}\left\{\begin{array}{l}u_t- \nu\Delta u+Uu_x+vU_y+p_x=-\varepsilon\kappa\Delta\varphi_\infty \varphi_x,\\
v_t- \nu\Delta v+Uv_x+p_y=-\varepsilon\kappa\Delta\varphi_\infty\varphi_y-\varepsilon\kappa(\varphi_\infty)_y\Delta\varphi,\\
w_t- \nu\Delta w+Uw_x+p_z=-\varepsilon\kappa\Delta\varphi_\infty\varphi_z,\\
u_x+v_y+w_z=0,\\
\varphi_t+U\varphi_x+(\varphi_\infty)_y v+\rho_0\varepsilon\Delta^2\varphi-\gamma\Delta\varphi=0,\\
t>0,\ x\in\mathbb{R},\ y\in(0,1),\ z\in\mathbb{R},\end{array}\right.\ \end{equation}supplemented with the B.C. and initial data \eqref{e7}. Notice that we have replaced the target concentration by its antisymmetric part $\varphi_\infty$. This fact is crucial in the proof of the key Lemma \ref{l2} below. Moreover, in \eqref{e9}, the quantity $\rho_0\alpha F''(\varphi_{tg})$ has been replaced by its mean value $\gamma$. Concerning the linearized system, we shall prove the following
\begin{theorem}\label{t1}Under assumptions $(H_0)$ and $(H_{11}),(H_{22})$ below, once one plugs the feedback controllers 
\begin{equation}\label{e180}\mathbf{S}(t,x,z):=\sum_{\sqrt{k^2+l^2}\leq M,\ k\neq0}\mathbf{s}_{kl}(t)e^{\text{i}kx}e^{\text{i}lz} \text{ and } \mathbf{T}(t,x,z):=\sum_{\sqrt{k^2+l^2}\leq M,\ k\neq0}\mathbf{t}_{kl}(t)e^{\text{i}kx}e^{\text{i}lz},\end{equation}into system \eqref{e9} it yields the existence of some $C,\eta>0$ such that the corresponding solution of the closed-loop system \eqref{e9} satisfies the exponential decay
\begin{equation}\label{e181}\|(u(t),v(t),w(t),\varphi(t))\|^2\leq Ce^{-\eta t}\|(u^o,v^o,w^o,\varphi^o)\|^2,\ \forall t\geq0.\end{equation} Here, $M>0$ is given by Lemmas \ref{l1}, \ref{cl1} below. Moreover, for $\sqrt{k^2+l^2}\leq M,\ k\neq0,l\neq0$, 
\begin{flalign}\label{te60}\mathbf{s}_{kl}(t):=-\frac{i}{k}\Omega_{kl}(t) \text{ and }\mathbf{t}_{kl}=\overline{a}\frac{\text{i}}{k}\Omega_{kl}(t),\end{flalign}
where $a\in\mathbb{C}$ is given by Lemma \ref{l2} below, and  $\Omega_{kl}=\Omega_{kl}(t,v,\varphi)$ has the form
\begin{flalign}\label{te61}\begin{aligned}&\Omega_{kl}(t):=\\&
\left<\Lambda_{sum}^{kl}\mathbf{R}_{kl}\left(\begin{array}{c}\int_\mathcal{O}\left[-v_{yy}+(k^2+l^2)v\right]\overline{Z_1^{kl*}}+\varphi\overline{\Phi_1^{kl*}}]e^{-\text{i}kx}e^{-\text{i}lz}dxdydz\\
\int_\mathcal{O}\left[-v_{yy}+(k^2+l^2)v\right]\overline{Z_2^{kl*}}+\varphi\overline{\Phi_2^{kl*}}]e^{-\text{i}kx}e^{-\text{i}lz}dxdydz\\
\ddots\\
\int_\mathcal{O}\left[-v_{yy}+(k^2+l^2)v\right]\overline{Z_{N_{kl}}^{kl*}}+\varphi\overline{\Phi_{N_{kl}}^{kl*}}]e^{-\text{i}kx}e^{-\text{i}lz}dxdydz\end{array}\right),\left(\begin{array}{c}(Z_1^{kl*})_{yy}(0)+a(Z_1^{kl*})_{yy}(1)\\ (Z_2^{kl*})_{yy}(0)+a(Z_2^{kl*})_{yy}(1)\\ \ddots \\ (Z_{N_{kl}}^{kl*})_{yy}(0)+a(Z_{N_{kl}}^{kl*})_{yy}(1)\end{array}\right)\right>_{N_{kl}}\end{aligned}&& \end{flalign}Here, $\Lambda_{sum}^{kl}$ are the diagonal matrices given by \eqref{e63} below, and $\mathbf{R}_{kl}$ are the square matrices given by \eqref{e66} below. $\mathcal{O}=(0,2\pi)\times(0,1)\times(0,2\pi),$ and $\left\{(Z_j^{kl*},\Phi_j^{kl*})\right\}_{j=1}^{N_{kl}}$ are  the first $N_{kl}\in\mathbb{N}$ eigenvectors of the  dual  of the operator $\mathbf{A}_{kl}$, given by \eqref{e25} below. $\left<\cdot,\cdot\right>_{N}$ stands for the scalar product in $\mathbb{C}^N$.

Furthermore, for $0<|k|\leq M$ \begin{flalign}\label{tce60}\mathbf{s}_{k0}(t):=-\frac{i}{k}\Omega_{k0}(t) \text{ and }\mathbf{t}_{k0}:=\overline{b}\frac{i}{k}\Omega_{k0},\end{flalign}Here $b\in\mathbb{C}$ is given by Lemma \ref{cl2} below, and
 $\Omega_{k0}=\Omega_{k0}(t,v,\varphi)$ has the form
\begin{flalign}\label{tce61}\begin{aligned}&\Omega_{k0}(t):=\\&
\left<\Lambda_{sum}^{k0}\mathbf{R}_{k0}\left(\begin{array}{c}\int_\mathcal{O}\left[-v_{yy}+k^2v\right]\overline{Z_1^{k0*}}+\varphi\overline{\Phi_1^{k0*}}]e^{-\text{i}kx}dxdydz\\
\int_\mathcal{O}\left[-v_{yy}+k^2v\right]\overline{Z_2^{k0*}}+\varphi\overline{\Phi_2^{k0*}}]e^{-\text{i}kx}dxdydz\\
\ddots\\
\int_\mathcal{O}\left[-v_{yy}+k^2v\right]\overline{Z_{N_{k0}}^{k0*}}+\varphi\overline{\Phi_{N_{k0}}^{k0*}}]e^{-\text{i}kx}dxdydz\end{array}\right),\left(\begin{array}{c}(Z_1^{k0*})_{yy}(0)+b(Z_1^{k0*})_{yy}(1)\\ (Z_2^{k0*})_{yy}(0)+b(Z_2^{k0*})_{yy}(1)\\ \ddots \\ (Z_{N_{k0}}^{k0*})_{yy}(0)+b(Z_{N_{k0}}^{k0*})_{yy}(1)\end{array}\right)\right>_{N_{k0}},
\end{aligned}&&\end{flalign}Here, $\Lambda_{sum}^{k0}$ are the diagonal matrices given by \eqref{ce63} below, and $\mathbf{R}_{k0}$ are the square matrices given by \eqref{ce66} below; and $\left\{(Z_j^{k0*},\Phi_j^{k0*})\right\}_{j=1}^{N_{k0}}$ are  the first $N_{k0}\in\mathbb{N}$ eigenvectors of the  dual  of the operator $\mathbf{A}_{k0}$, given by \eqref{ce25} below.
\end{theorem}

Roughly speaking, the approach will be the following: in virtue of the $2\pi-$ periodicity assumption, we shall decompose the linear system \eqref{e9} in Fourier modes, obtaining so an infinite  system parametrized by $(k,l)\in\mathbb{Z}\times\mathbb{Z}$. Thus, stabilization of the linear system \eqref{e9} will be equivalent with the stabilization of the infinite system, at each level $(k,l)$. We shall consider different cases of $(k,l)$, and reduce the pressure from the equations, then we shall implement the control design method described in \cite[Chapter 2]{book}. This method has been successfully applied before for the Navier-Stokes equations, Magnetohydrodyanmics equations, the Cahn-Hilliard system, stochastic PDEs,  see \cite{book}.  The idea has its origins in \cite{bsb,ion2}.

\section{Fourier decomposition of the linear system}
Taking advantage of the periodic assumption, we shall place ourselves in the Fourier functional setting. More precisely, we set $L^2_{2\pi}(\mathcal{O})$ to stand for the space containing all the functions $u\in L^2_{loc}(\mathbb{R}\times(0,1)\times\mathbb{R})$ which are $2\pi-$periodic in both $x$ and $z$ directions. The space $L^2_{loc}(\mathbb{R}\times(0,1)\times\mathbb{R})$ consists of all functions whose square is locally Lebesgue integrable on $\mathbb{R}\times(0,1)\times\mathbb{R}$.  For a function $u\in L^2_{2\pi}(\mathcal{O})$ , a \textit{Fourier decomposition} can be considered, namely
$$u(x,y,z)=\sum_{k,l\in\mathbb{Z}}u_{kl}(y)e^{\text{i}kx}e^{\text{i}lz},$$with
$$\sum_{k,l\in\mathbb{Z}}\int_0^1|u_{kl}(y)|^2dy<\infty.$$ The coefficients $u_{kl}$ are called the Fourier modes, and in order to assure the fact that $u$ is real, the following must hold
$$u_{kl}=\overline{u_{-k-l}},\ \forall k,l\in\mathbb{Z}.$$ (Here $\overline{u}$ stands for the complex conjugate of $u\in\mathbb{C}$.) Besides this, the norm in $L^2_{2\pi}(\mathcal{O})$ is defined as
\begin{flalign}\label{e16}\|u\|_{L^2_{2\pi}(\mathcal{O})}=2\pi\left(\sum_{k,l\in\mathbb{Z}}\|u_{kl}\|^2_{L^2(0,1)}\right)^\frac{1}{2}.\end{flalign} 

  Taking into account that some of the operators, we shall deal with  below, might have complex eigenvalues, it will be convenient in the sequel to view a linear operator $A$ as a linear operator (again denoted by $A$) in the complexified space $H=L^2(0,1)+\text{i}L^2(0,1).$ We denote by $\left<\cdot,\cdot\right>$ the scalar product in $H$ and by $\|\cdot\|$ its norm. We shall denote by $H^m(0,1),\ m=1,2,...,$ the standard Sobolev spaces on $(0,1)$, with pivot space $H$, and by
$$H_0^1(0,1):=\left\{v\in H^1(0,1):\ v(0)=v(1)=0\right\},$$
$$H_0^2(0,1):=\left\{v\in H^2(0,1)\cap H_0^1(0,1):\ v'(0)=v'(1)=0\right\}.$$Here, for a function $v:[0,1]\rightarrow\mathbb{C} $ we denote by $v'$ its first derivative. It is well-known the Poincar\' e inequality for a function $f\in H_0^1(0,1)$:
\begin{equation}\label{P} 2\|f\|^2\leq \|f'\|^2.\end{equation}

Below, we shall need to work with the product spaces $H\times H$ or $H\times H\times H\times H$ as-well. Since there is no danger of confusion, we still denote by $\left<\cdot,\cdot\right>$ and by $\|\cdot\|$ the scalar product and the corresponding norm of those spaces, respectively. The difference will be clear from the context. 

$$$$

\noindent\textit{\textbf{Proof of Theorem \ref{t1}.}} We recall system \eqref{e9}. We decompose it in  the Fourier modes $\left\{u_{kl},v_{kl},w_{kl},\varphi_{kl},p_{kl},\mathbf{s}_{kl},\mathbf{t}_{kl}\right\}_{k,l}$ of $u,v,w,\varphi,p,\mathbf{S},\mathbf{T},$ respectively. We get
\begin{flalign}\label{e11}\left\{\begin{array}{l}(u_{kl})_t- \nu[-(k^2+l^2)u_{kl}+u''_{kl}]+\text{i}kUu_{kl}+U'v_{kl}+\text{i}kp_{kl}=-\text{i}k\varepsilon\kappa\Delta\varphi_\infty\varphi_{kl},\\
(v_{kl})_t- \nu[-(k^2+l^2)v_{kl}+v''_{kl}]+\text{i}kUv_{kl}+p'_{kl}=-\varepsilon\kappa\Delta\varphi_\infty\varphi'_{kl}-\varepsilon\kappa\varphi'_\infty(-(k^2+l^2)\varphi_{kl}+\varphi''_{kl}),\\
(w_{kl})_t- \nu[-(k^2+l^2)w_{kl}+w''_{kl}]+\text{i}kUw_{kl}+\text{i}lp_{kl}=-\text{i}l\varepsilon\kappa\Delta\varphi_\infty\varphi_{kl},\\
\text{i}ku_{kl}+v'_{kl}+\text{i}lw_{kl}=0,\\
\begin{aligned}(\varphi_{kl})_t+&\varphi'_\infty v_{kl}+\rho_0\varepsilon\varphi_{kl}^{iv}-\left[2\rho_0\varepsilon(k^2+l^2)+\gamma\right]\varphi''_{kl}\\& +\left[\rho_0\varepsilon(k^2+l^2)^2+\gamma(k^2+l^2)
+\text{i}kU\right]\varphi_{kl}=0,\text{ a.e. in }(0,1),\end{aligned}\\
u_{kl}(0)=\mathbf{s}_{kl},\ u_{kl}(1)=\mathbf{t}_{kl},\ v_{kl}(0)=v_{kl}(1)=w_{kl}(0)=w_{kl}(1)=0,\\
\varphi'_{kl}(0)=\varphi'_{kl}(1)=\varphi'''_{kl}(0)=\varphi'''_{kl}(1)=0.\end{array}\right.\ &&\end{flalign}

Clearly, the asymptotic exponential stability of the linear system \eqref{e9} is equivalent with the stability of \eqref{e11} at each level $(k,l)\in\mathbb{Z}\times\mathbb{Z}$, with the coefficients of the exponential decay independent of the level. Hence, in what follows, we shall consider different cases for the couple $(k,l)$ (which  cover all the possibilities), then design stabilizers at each level. Finally, we shall conclude with the stability result for the linearized system \eqref{e9}. Firstly, let us consider  the most complex case, namely: 

\noindent\textbf{1. The case $k\neq 0$ and $l\neq 0$.} We reduce the pressure from \eqref{e11}, in the next manner: we add the
derivative of the first equation in \eqref{e11}, multiplied by $\text{i}k$, to the derivative
of the third equation in \eqref{e11}, multiplied by $\text{i}l$ , and to the second
equation in \eqref{e11}, multiplied by $(k^2 + l^2)$, then use the free divergence relation to get that
\begin{flalign}\label{e12}\left\{\begin{array}{l}\begin{aligned} & [-v''_{kl} + (k^2+l^2)v_{kl}]_t +  \nu v^{iv}_{kl}-[2 \nu(k^2+l^2)+\text{i}kU]v''_{kl}\\&+[ \nu(k^2+l^2)^2+\text{i}k(k^2+l^2)U+\text{i}kU'']v_{kl}-
\varepsilon\kappa(k^2+l^2)\varphi'''_\infty\varphi_{kl}\\&+(k^2+l^2)\varepsilon\kappa\varphi'_\infty(-(k^2+l^2)\varphi_{kl}+\varphi''_{kl})=0,\end{aligned}\\
\\
\begin{aligned}&(\varphi_{kl})_t+\varphi'_\infty v_{kl}+\rho_0\varepsilon\varphi_{kl}^{iv}-\left[2\rho_0\varepsilon(k^2+l^2)+\gamma\right]\varphi''_{kl}\\&
+\left[\rho_0\varepsilon(k^2+l^2)^2+\gamma(k^2+l^2)+\text{i}kU\right]\varphi_{kl}=0,
\text{ a.e. in }(0,1),\end{aligned}\\
\\
v_{kl}(0)=v_{kl}(1)=0, \ v'_{kl}(0)=-\text{i}k\ \mathbf{s}_{kl},v'_{kl}(1)=-\text{i}k\ \mathbf{t}_{kl},\\
\varphi'_{kl}(0)=\varphi'_{kl}(1)=\varphi'''_{kl}(0)=\varphi'''_{kl}(1)=0.\end{array}\right.\ &&\end{flalign}

Next, we direct our effort to write the equations \eqref{e12} in an abstract form. To this end, for each $k,l\in\mathbb{Z}\setminus\left\{0\right\}$, we denote by $L_{kl}:\mathcal{D}(L_{kl})\subset H\rightarrow H,$ $F_{kl}:\mathcal{D}(F_{kl})\subset H\rightarrow H $ and $E_{kl}:\mathcal{D}(E_{kl})\subset H\rightarrow H$ the operators
\begin{flalign}\label{e13}\begin{aligned}& L_{kl}v:=-v''+(k^2+l^2)v,
\forall v\in \mathcal{D}(L_{kl})=H_0^1(0,1)\cap H^2(0,1);\end{aligned}&&\end{flalign}
\begin{flalign}\label{e14}\begin{aligned}& F_{kl}v:= \nu v^{iv}-[2\nu (k^2+l^2)+\text{i}kU]v''+[\nu (k^2+l^2)^2+\text{i}k(k^2+l^2)U+\text{i}kU'']v,\\&
\forall v\in \mathcal{D}(F_{kl})=H_0^2(0,1)\cap H^4(0,1);\end{aligned}&&\end{flalign}and
\begin{flalign}\label{e15}\begin{aligned}& E_{kl}\varphi:=\rho_0\varepsilon\varphi^{iv}-\left[2\rho_0\varepsilon(k^2+l^2)+\gamma\right]\varphi''+\left[\rho_0\varepsilon(k^2+l^2)^2+\gamma(k^2+l^2)+\text{i}kU\right]\varphi,\\&
\forall \varphi\in \mathcal{D}(E_{kl})=\left\{\varphi\in H^4(0,1):\ \varphi'(0)=\varphi'(1)=\varphi'''(0)=\varphi'''(1)=0\right\}.\end{aligned}&&\end{flalign} In addition, for latter purpose, we consider as-well their differential forms. That is
\begin{flalign}\label{e19}\begin{aligned}& \mathcal{L}_{kl}v:=-v''+(k^2+l^2)v\end{aligned}&&\end{flalign}
\begin{flalign}\label{e20}\begin{aligned}& \mathcal{F}_{kl}v:= \nu v^{iv}-[2\nu (k^2+l^2)+\text{i}kU]v''+[ \nu(k^2+l^2)^2+\text{i}k(k^2+l^2)U+\text{i}kU'']v,\end{aligned}&&\end{flalign}and
\begin{flalign}\label{e21}\begin{aligned}& \mathcal{E}_{kl}\varphi:=\rho_0\varepsilon\varphi^{iv}-\left[2\rho_0\varepsilon(k^2+l^2)+\gamma\right]\varphi''+\left[\rho_0\varepsilon(k^2+l^2)^2+\gamma(k^2+l^2)+\text{i}kU\right]\varphi.\end{aligned}&&\end{flalign}With these notations, system \eqref{e12}, can be equivalently expressed as
\begin{flalign}\label{e23}\left\{\begin{array}{l}(\mathcal{L}_{kl}v_{kl})_t+\mathcal{F}_{kl}v_{kl}-\varepsilon\kappa(k^2+l^2)\varphi'_\infty\mathcal{L}_{kl}\varphi_{kl}-\varepsilon\kappa(k^2+l^2)\varphi'''_\infty\varphi_{kl}=0,\\
(\varphi_{kl})_t+\varphi'_\infty v_{kl}+\mathcal{E}_{kl}\varphi_{kl}=0,\ \text{a.e. in } (0,1),\\
v'_{kl}(0)=-\text{i}k\mathbf{s}_{kl},v'_{kl}(1)=-\text{i}k\mathbf{t}_{kl},\ v_{kl}(0)=v_{kl}(1)=0,\\
\varphi'_{kl}(0)=\varphi'_{kl}(1)=\varphi'''_{kl}(0)=\varphi'''_{kl}(1)=0.
\end{array}\right.\ &&\end{flalign}
Aiming to further improve the expression of the system \eqref{e23}, we  introduce the operator $\mathbb{A}_{kl}:\mathcal{D}(\mathbb{A}_{kl})\subset H\rightarrow H,$ as
\begin{flalign}\label{e26}\mathbb{A}_{kl}:=F_{kl}L^{-1}_{kl},\ \mathcal{D}(\mathbb{A}_{kl})=\left\{z\in H:\ L^{-1}_{kl}z\in\mathcal{D}(F_{kl})\right\}.\end{flalign} Then, define $\mathbf{A}_{kl}:\mathcal{D}(\mathbf{A}_{kl})\subset H\times H\rightarrow H\times H,$ as
\begin{flalign}\label{e25}\mathbf{A}_{kl}\left(\begin{array}{c}z \\ \varphi\end{array}\right):=\left(\begin{array}{ccc}\mathbb{A}_{kl} & & -\varepsilon\kappa(k^2+l^2)\varphi'_\infty L_{kl}-\varepsilon\kappa(k^2+l^2)\varphi'''_\infty\\
\varphi'_\infty L_{kl}^{-1} & & E_{kl}\end{array}\right)\left(\begin{array}{c}z \\  \varphi\end{array}\right),\end{flalign}for all
$$\left(\begin{array}{c}z \\ \varphi\end{array}\right)\in \mathcal{D}(\mathbf{A}_{kl})=\mathcal{D}(\mathbb{A}_{kl})\times \mathcal{D}(E_{kl}).$$ (We note that, in the definition of $\mathbf{A}_{kl}$, the operator $L_{kl}$, in the upper right corner, is understood with the domain $\left\{\varphi\in H^2(0,1): \varphi'(0)=\varphi'(1)=0\right\}$).  
Concerning $\mathbf{A}_{kl}$, we shall prove the following two paramount results.
\begin{lemma}\label{l1}The operator $-\mathbf{A}_{kl}$ generates a $C_0-$analytic semigroup on $\mathcal{D}(L_{kl}^{-1})\times H$, and for each $\lambda\in\rho(-\mathbf{A}_{kl})$ (the resolvent set of $-\mathbf{A}_{kl}$), $(\lambda I+\mathbf{A}_{kl})^{-1}$ is compact. Moreover, one has for each $\eta>0$ there exists $M>0$, sufficiently large such that
\begin{flalign}\label{e30}\begin{aligned}&\sigma(-\mathbf{A}_{kl})\subset\left\{\lambda\in\mathbb{C}:\ \Re\lambda\leq -\eta\right\},\forall \sqrt{k^2+l^2}> M,\end{aligned}\end{flalign}where $\sigma(-\mathbf{A}_{kl})$ is the spectrum of $-\mathbf{A}_{kl}$.
\end{lemma}
\begin{proof} For each $\lambda\in\mathbb{C}$ and $(f,g)\in \mathcal{D}(L_{kl}^{-1})\times H$, consider the equation
$$\lambda \left(\begin{array}{c}z \\ \varphi\end{array}\right)+\mathbf{A}_{kl}\left(\begin{array}{c}z \\ \varphi\end{array}\right)=\left(\begin{array}{c}f \\ g\end{array}\right)$$
or, equivalently
\begin{flalign}\label{e31}\left\{\begin{array}{l}\lambda L_{kl}v+F_{kl}v-\varepsilon\kappa(k^2+l^2)\varphi'_\infty L_{kl}\varphi-\varepsilon\kappa(k^2+l^2)\varphi'''_\infty\varphi=f,\\
\lambda \varphi +\varphi'_\infty v+E_{kl}\varphi=g.\end{array}\right.\ \end{flalign}
Taking into account \eqref{e19}-\eqref{e21} and the fact that $E_{kl}\varphi$ can be rewritten in the equivalent form as
$$E_{kl}\varphi=-\rho_0\varepsilon(L_{kl}\varphi)''+[\rho_0\varepsilon(k^2+l^2)+\gamma]L_{kl}\varphi+\text{i}kU\varphi,$$
 if we scalarly multiply the first equation of \eqref{e31}  by $v$, then the second equation of \eqref{e31} by $L_{kl}\varphi$, and take the real part of the result, it yields that
\begin{flalign}\label{e32}\left\{\begin{array}{l}\begin{aligned}&\Re\lambda[\|v'\|^2+(k^2+l^2)\|v\|^2]+ \nu\|v''\|^2+2 \nu(k^2+l^2)\|v'\|^2+\nu (k^2+l^2)^2\|v\|^2\\&
+k\int_0^1U'(\Re v'\Im v-\Im v'\Re v)dy=\varepsilon\kappa(k^2+l^2)\Re\left<\varphi'_\infty L_{kl}\varphi,v\right>\\&+\Re\left<\varepsilon\kappa(k^2+l^2)\varphi'''_\infty\varphi_{kl},v\right>+\Re\left<f,v\right>\end{aligned}\\
\\
\begin{aligned}&\Re\lambda [\|\varphi'\|^2+(k^2+l^2)\|\varphi\|^2]+\rho_0\varepsilon\|(L_{kl}\varphi)'\|^2+[\rho_0\varepsilon(k^2+l^2)+\gamma]\|L_{kl}\varphi\|^2\\&
=-\Re\left<\varphi'_\infty v,L_{kl}\varphi\right>-\Re\left(\text{i}k\left<U\varphi,L_{kl}\varphi\right>\right)+\Re\left<g,L_{kl}\varphi\right>.\end{aligned}
\end{array}\right.\ &&\end{flalign}By Young's inequality,  the fact that $U$ is bounded and all the derivatives of $\varphi_\infty$ are bounded,  it immediately follows from the summation of the two equations of \eqref{e32}, that
\begin{flalign}\label{e33}\begin{aligned}&\Re\lambda[\|v'\|^2+(k^2+l^2)\|v\|^2+\|\varphi'\|^2+(k^2+l^2)\|\varphi\|^2]
+\nu \|v''\|^2+2\nu (k^2+l^2)\|v'\|^2\\& +\nu(k^2+l^2)^2\|v\|^2
+\rho_0\varepsilon\|(L_{kl}\varphi)'\|^2+[\rho_0\varepsilon(k^2+l^2)+\gamma]\|L_{kl}\varphi\|^2\\&
\leq \rho_0\varepsilon(k^2+l^2)\|L_{kl}\varphi\|^2+C_1(k^2+l^2)\|v\|^2+
\nu k^2\|v'\|^2
+ \frac{1}{2}\gamma\|L_{kl}\varphi\|^2\\&+C_2(\|v\|^2+\|\varphi\|^2)+C_3(\|f\|^2+\|g\|^2),\end{aligned}&&\end{flalign}where $C_i,i=1,2,3$ are some positive constants independent of $k$ or $l$. We conclude from this that
$$\|(L_{kl}^{-1}z,\varphi)\|_{H\times H}\leq \frac{C}{\lambda-\omega}\|(f,g)\|_{H\times H},\ \forall \lambda >\omega,$$ for some $\omega>0$ and $C>0$. Equivalently
$$\|(\lambda I+\mathbf{A}_{kl})^{-1}(f,g)\|_{\mathcal{D}(L_{kl}^{-1})\times H}\leq\frac{C}{\lambda-\omega}\|(f,g)\|_{H\times H},\ \forall \lambda >\omega.$$Hence, the Hille-Yosida theorem assures that $- \mathbf{A}_{kl}$ is the infinitesimal generator of a $C_0-$analytic semigroup, $\left\{e^{-\mathbf{A}_{kl}t},\ t\geq0\right\}$. 

Now, let $\lambda\in\mathbb{C}$ and $(z_{kl},\varphi_{kl})\in H\times H$, such that  $\lambda(z_{kl},\varphi_{kl})+\mathbf{A}_{kl}(z_{kl},\varphi_{kl})=0$. Then, similarly as in \eqref{e31}-\eqref{e33}, we get that
\begin{flalign}\label{ie40}\begin{aligned}&\Re\lambda[\|v_{kl}'\|^2+(k^2+l^2)\|v_{kl}\|^2+\|\varphi_{kl}'\|^2+(k^2+l^2)\|\varphi_{kl}\|^2] \\&
+[ \nu(k^2+l^2)^2-C_1(k^2+l^2)]\|v_{kl}\|^2+\frac{1}{2}\gamma\|L_{kl}\varphi_{kl}\|^2-C_2\|\varphi_{kl}\|^2\leq 0, \forall k,l\in\mathbb{Z}\setminus\left\{0\right\}. \end{aligned}&&\end{flalign}Here, $L_{kl}v_{kl}=z_{kl}.$
Thus, for each $\eta>0$, \eqref{ie40} says that if we take $k^2+l^2$ sufficiently large (for a large $M>0$, we take $\sqrt{k^2+l^2}> M$), then $\Re\lambda\leq-\eta$. This concludes the proof.
\end{proof} By Lemma \ref{l1} we see that if $k^2+l^2$ is large enough, then the operators $-\mathbf{A}_{kl}$ have  stable spectrum. One may conclude from this that the system \eqref{e11} is exponentially asymptotically stable when $k^2+l^2$ is large enough. This is indeed so. To see this, put everywhere null boundary conditions in the system \eqref{e11}. Next,  scalarly multiply the first, second, third and fifth equation of \eqref{e11} by $u_{kl},v_{kl},w_{kl}$ and $\varphi_{kl}$, respectively. Summing them  and taking  the real part of the result, we get that
$$\begin{aligned}& \frac{1}{2}\frac{d}{dt}\left(\|u_{kl}\|^2+\|v_{kl}\|^2+\|w_{kl}\|^2+\|\varphi_{kl}\|^2\right)+\nu(k^2+l^2)(\|u_{kl}\|^2+\|v_{kl}\|^2+\|w_{kl}\|^2)\\&
+\nu\|u'_{kl}\|^2+\nu\|v'_{kl}\|^2+\nu\|w'_{kl}\|^2+\rho_0\varepsilon\|\varphi_{kl}''\|^2\\&
+[2\rho_0\varepsilon(k^2+l^2)+\gamma]\|\varphi_{kl}'\|^2+[\rho_0\varepsilon(k^2+l^2)^2+\gamma(k^2+l^2)]\|\varphi_{kl}\|^2\\&
=\Re\left[\int_0^1U'v_{kl}\overline{u_{kl}}dy+\varepsilon\kappa(k^2+l^2)\int_0^1\varphi'_\infty\varphi_{kl}\overline{v_{kl}}dy-\varepsilon\kappa\int_0^1\varphi'_\infty\varphi_{kl}''\overline{v_{kl}}dy-\int_0^1\varphi'_\infty v_{kl}\overline{\varphi_{kl}}dy\right]\\&
+\Re\int_0^1 \varepsilon\kappa(k^2+l^2)\varphi'''_\infty\varphi_{kl}\overline{v_{kl}}dy.\end{aligned}$$
Again using conveniently Young's inequality, we get from above that
$$\begin{aligned}& \frac{1}{2}\frac{d}{dt}\left(\|u_{kl}\|^2+\|v_{kl}\|^2+\|w_{kl}\|^2+\|\varphi_{kl}\|^2\right)+\nu(k^2+l^2)(\|u_{kl}\|^2+\|v_{kl}\|^2+\|w_{kl}\|^2)\\&
+\nu\|u'_{kl}\|^2+\nu\|v'_{kl}\|^2+\nu\|w'_{kl}\|^2+\rho_0\varepsilon\|\varphi_{kl}''\|^2\\&
+[2\rho_0\varepsilon(k^2+l^2)+\gamma]\|\varphi_{kl}'\|^2+[\rho_0\varepsilon(k^2+l^2)^2+\gamma(k^2+l^2)]\|\varphi_{kl}\|^2\\&
\leq \rho_0\varepsilon(k^2+l^2)^2\|\varphi_{kl}\|^2+\rho_0\varepsilon\|\varphi_{kl}''\|^2+C(\|u_{kl}\|^2+\|v_{kl}\|^2+\|w_{kl}\|^2+\|\varphi_{kl}\|^2),
\end{aligned}$$where $C$ is some positive constant independent of $k$ and $l$. It follows from above that
$$\begin{aligned}& \frac{1}{2}\frac{d}{dt}\left(\|u_{kl}\|^2+\|v_{kl}\|^2+\|w_{kl}\|^2+\|\varphi_{kl}\|^2\right)\\&
+\min\left\{\gamma(k^2+l^2),\nu(k^2+l^2)\right\}(\|u_{kl}\|^2+\|v_{kl}\|^2+\|w_{kl}\|^2+\|\varphi_{kl}\|^2)\\&
\leq C(\|u_{kl}\|^2+\|v_{kl}\|^2+\|w_{kl}\|^2+\|\varphi_{kl}\|^2).\end{aligned}$$Hence, if $\sqrt{k^2+l^2}>M$, with $M>0$ large enough, we have that there exist $C_1,\eta_1>0$, independent of $k$ or $l$, such that
\begin{flalign}\label{e59}\|(u_{kl}(t),v_{kl}(t),w_{kl}(t),\varphi_{kl}(t))\|^2\leq C_1e^{-\eta_1t}(\|(u_{kl}(0),v_{kl}(0),w_{kl}(0),\varphi_{kl}(0)\|^2,\ \forall k^2+l^2>M.\end{flalign}
Therefore, we have to control the system \eqref{e11} for $\sqrt{k^2+l^2}\leq M$ only.

Furthermore, Lemma \ref{l1} guarantees that, for each couple $(k,l)$, $-\mathbf{A}_{kl}$ has a countable set of eigenvalues, denoted by $\left\{\lambda_j^{kl}\right\}_{j=1}^\infty$ (we repeat each $\lambda$ according to its multiplicity). Moreover, there is only a finite number $N_{kl}\in\mathbb{N}$ of eigenvalues with $\Re \lambda_j^{kl}\geq0,$ which are usually called the unstable eigenvalues. Let us denote by $\left\{\left(Z_j^{kl}, \Phi_j^{kl}\right)\right\}_{j=1}^\infty$ and by $\left\{\left(Z_j^{kl*}, \Phi_j^{kl*}\right)\right\}_{j=1}^\infty$ the corresponding eigenvectors system of $-\mathbf{A}_{kl}$ and of its dual $-\mathbf{A}^*_{kl},$ respectively. Here and below, we  denote by $(\cdot,\cdot)$ or by $\left(\begin{array}{c}\cdot\\ \cdot\end{array}\right)$ vectors in $H\times H$.

We assume that the following assumption holds:
$$(H_1)\text{ All the unstable eigenvlues $\lambda_{j}^{kl},\ \sqrt{k^2+l^2}\leq M,\ j=1,2,..., N_{kl},$ are semisimple.}$$
This means that, for of all $\lambda_j^{kl},$ the geometric multiplicity  coincides with the algebraic multiplicity. We notice that, for the case of the Navier-Stokes equations it is shown in \cite{9,8}  that the property $(H_1)$, or more generally that $\lambda_{j}$ are simple eigenvalues, is generic with respect to the coefficients. Similar results can be obtained for the Cahn-Hilliard-Navier-Stokes equations.

The next key lemma  shows a unique  continuation type result for the eigenvectors of the dual operator, corresponding to the unstable eigenvalues. It is a classical tool when dealing with boundary controllers, and it represents the main ingredient for the control design here. 

\begin{lemma}\label{l2} Under assumption $(H_1)$, there exists an $a\in\mathbb{C}$ such that  for each unstable eigenvalue $\overline{\lambda_j^{kl}}$ $(\sqrt{k^2+l^2}\leq M$ with $M$ from Lemma \ref{l1}; and $j\in\left\{1,2,...,N_{kl}\right\}),$ of the dual operator $-\mathbf{A}_{kl}^*$, one can choose the corresponding eigenvector $\left(Z_j^{kl*}, \Phi_j^{kl*}\right)$   in such a way that
\begin{flalign}\label{e40} (Z_{j}^{kl*})''(0)+a(Z_j^{kl*})''(1)\neq0.\end{flalign}
\end{lemma}
\begin{proof}In fact, we shall show that we can choose  the eigenvector such that either $(Z_j^{kl*})''(0)\neq0$ or $(Z_j^{kl*})''(1)\neq0.$ From this, in a straightforward manner, we shall conclude to the proof. Let us assume by contradiction that
$$(Z_j^{kl*})''(0)=(Z_j^{kl*})''(1)=0,$$from which we shall arrive to a contradiction.

Recall that, for a function $f:[0,1]\rightarrow \mathbb{C}$, we denote by $\hat{f}:[0,1]\rightarrow\mathbb{C}$:
$$\hat{f}(y):=f(1-y),\ y\in[0,1].$$We say that $f$ is symmetric if $f\equiv\hat{f}$, and antisymmetric if $f\equiv-\hat{f}$. Let us denote by 
$$\mathcal{S}=\left\{f:[0,1]\rightarrow\mathbb{C}:\ f\equiv \hat{f}\right\} \text{ and by }\mathcal{AS}=\left\{f:[0,1]\rightarrow\mathbb{C}:\ f\equiv -\hat{f}\right\}.$$

 Let  us recall relation \eqref{e25}, which gives the exact form the operator $\mathbf{A}_{kl}$. Hence, its dual has the following form
\begin{flalign}\label{e74}\mathbf{A}_{kl}^* \left(\begin{array}{c}z \\ \varphi\end{array}\right)=\left(\begin{array}{ccc}L_{kl}^{-1}F_{kl}^* & & L_{kl}^{-1}(\varphi'_\infty\cdot)\\  -\varepsilon\kappa(k^2+l^2) L_{kl}(\varphi'_\infty\cdot)-\varepsilon\kappa(k^2+l^2)\varphi'''_\infty  & & E^*_{kl}\end{array}\right)\left(\begin{array}{c}z \\  \varphi\end{array}\right),\end{flalign} where $F_{kl}^*$ and $E_{kl}^*$ are the dual operators of $F_{kl}$ and $E_{kl}$, respectively. Consequently, the relation  
$$\overline{\lambda_{kl}}\left(Z_j^{kl*}, \Phi_j^{kl*}\right)+\mathbf{A}_{kl}^*\left(Z_j^{kl*}, \Phi_j^{kl*}\right)=0$$ is equivalent with (for the ease of writing we shall replace $\lambda_{kl},\ Z_j^{kl*},\ \Phi_{j}^{kl*}$ simply by $\lambda, Z^*$ and $\Phi^*$, respectively. Also, we shall write $\mathcal{L},\mathcal{F},\mathcal{E}$, instead of $\mathcal{L}_{kl},\mathcal{F}_{kl},\mathcal{E}_{kl}$, respectively. )
\begin{flalign}\label{e42}\left\{\begin{array}{l}\begin{aligned}&  \nu (Z^*)^{iv}-(2 \nu(k^2+l^2)-\text{i}kU+\overline{\lambda})(Z^*)''-2\text{i}kU'(Z^*)'\\&
+[(k^2+l^2)\overline{\lambda}+\nu (k^2+l^2)^2-\text{i}k(k^2+l^2)U]Z^*+\varphi'_\infty\Phi^*=0, \end{aligned}\\ \\
\begin{aligned}&-\varepsilon \kappa(k^2+l^2)L_{kl}(\varphi'_\infty Z^*)-\varepsilon\kappa(k^2+l^2)\varphi'''_\infty Z^*
+\rho_0\varepsilon(\Phi^*)^{iv}-\left[2\rho_0\varepsilon(k^2+l^2)+\gamma\right](\Phi^*)''\\&+\left[\rho_0\varepsilon(k^2+l^2)^2+\gamma(k^2+l^2)+\overline{\lambda}-\text{i}kU\right]\Phi^*=0,\ \text{ a.e. in }(0,1),\end{aligned}\\
\\
 Z^*(0)=Z^*(1)=(Z^*)'(0)=(Z^*)'(1)=0=(Z^*)''(0)=(Z^*)''(1)=0,\\
(\Phi^*)'(0)=(\Phi^*)'(1)=(\Phi^*)'''(0)=(\Phi^*)'''(1)=0.\end{array}\right.\ \end{flalign}Readily seen, the solution of the system \eqref{e42} satisfies
$$Z^*\equiv 0 \Leftrightarrow \Phi^*\equiv 0.$$ This is indeed so, assume, for instance, that $Z^*\equiv 0$. Plugging this in the second equation of \eqref{e42}, it yields that
$$\rho_0\varepsilon(\Phi^*)^{iv}-\left[2\rho_0\varepsilon(k^2+l^2)+\gamma\right](\Phi^*)''+\left[\rho_0\varepsilon(k^2+l^2)^2+\gamma(k^2+l^2)+\overline{\lambda}-\text{i}kU\right]\Phi^*=0.$$Then, scalarly multiply it by $\Phi^*$ and taking the real part of the result, we obtain that
 $$\rho_0\varepsilon\|(\Phi^*)^{''}\|^2+\left[2\rho_0\varepsilon(k^2+l^2)+\gamma\right]\|(\Phi^*)'\|^2+\left[\rho_0\varepsilon(k^2+l^2)^2+\gamma(k^2+l^2)+\Re\lambda\right]\|\Phi^*\|^2=0,$$where using the fact that $\lambda$ is an unstable eigenvalue, i.e. $\Re\lambda\geq0$, we immediately get $\Phi^*\equiv0$. On the other hand, if $\Phi^*\equiv0$, plugging this in the second equation of \eqref{e42}, we deduce that the analytic function $Z^*$ satisfies a second order linear differential equation and has three null boundary conditions $Z^*(0)=(Z^*)'(0)=(Z^*)''(0)=0$. Clearly, this implies that $Z^*\equiv0$. 

Since $U=\hat{U}$ and $\varphi_\infty=-\hat{\varphi_\infty}$, it is easy to see that if $(Z^*,\Phi^*)$ satisfies \eqref{e42}, then $(\hat{Z^*},\hat{U^*})$ satisfies \eqref{e42} as-well.  Consequently, $(Z^*+\hat{Z^*}, \Phi^*+\hat{\Phi^*})$ and $(Z^*-\hat{Z^*},\Phi^*-\hat{\Phi^*})$ satisfy \eqref{e42}. Moreover, it is obvious that 
$$Z^*\in \mathcal{S} \Leftrightarrow \Phi^*\in\mathcal{S} \text{ and }Z^*\in \mathcal{AS} \Leftrightarrow \Phi^*\in\mathcal{AS}.$$ Indeed, assume, for instance, that $Z^*\in \mathcal{S}$, then $(Z^*-\hat{Z^*},\Phi^*-\hat{\Phi^*})$ satisfies \eqref{e42}. Since $Z^*-\hat{Z^*}\equiv 0$, it follows from above that $\Phi^*-\hat{\Phi^*}\equiv 0$, namely $\Phi^*\in\mathcal{S}$ (the other cases can be treated similarly). 

Below, all our effort is to show that the additional hypothesis $(Z^*)''(0)=(Z^*)''(1)=0$ implies that $Z^*\equiv 0$, and consequently $\Phi^*\equiv0$, which   contradicts  the fact that $(Z^*,\Phi^*)$ is an eigenvector. In order to do that,  for the beginning we shall work under an additional assumption, namely $Z^*$ has a symmetry (and consequently $\Phi^*$ has the same symmetry.) From this it will yield that $Z^*\equiv \Phi^*\equiv0$. Next, in case that $Z^*$ has no symmetry, we shall replace $(Z^*,\Phi^*)$ by $(Z^*+\hat{Z^*},\Phi^*+\hat{\Phi^*})$. Note that $(Z^*+\hat{Z^*},\Phi^*+\hat{\Phi^*})$ satisfies \eqref{e42}, 
$(Z^*+\hat{Z^*})''(0)=(Z^*+\hat{Z^*})''(1)=0,$ and that $(Z^*+\hat{Z^*},\Phi^*+\hat{\Phi^*})$ is symmetric. Relying on the previous results, it will follow that $(Z^*+\hat{Z^*},\Phi^*+\hat{\Phi^*})\equiv (0,0)$, and consequently $Z^*$ is antisymmetric.But, this contradicts the fact that  $Z^*$ has no symmetry.  Therefore, in order to conclude with the proof, it is enough to consider only the case when $Z^*$ has a symmetry. To fix the ideas, assume for instance that $Z^*$ is symmetric, consequently $\Phi^*$ is symmetric as-well (the other case can be treated similarly).

Let us denote by 
$$\mu_1:= \varphi'_\infty\Phi^* \text{ and by } \mu_2:=-\varepsilon \kappa(k^2+l^2)L_{kl}(\varphi'_\infty Z^*)-\varepsilon\kappa(k^2+l^2)\varphi'''_\infty Z^*.$$ With these notations, system \eqref{e42} reads as
\begin{flalign}\label{e137}\left\{\begin{array}{l}\begin{aligned}&  \nu (Z^*)^{iv}-(2 \nu(k^2+l^2)-\text{i}kU+\overline{\lambda})(Z^*)''-2\text{i}kU'(Z^*)'\\&
+[(k^2+l^2)\overline{\lambda}+\nu (k^2+l^2)^2-\text{i}k(k^2+l^2)U]Z^*+\mu_1=0, \end{aligned}\\
\\
\begin{aligned}&\mu_2
+\rho_0\varepsilon(\Phi^*)^{iv}-\left[2\rho_0\varepsilon(k^2+l^2)+\gamma\right](\Phi^*)''\\&+\left[\rho_0\varepsilon(k^2+l^2)^2+\gamma(k^2+l^2)+\overline{\lambda}-\text{i}kU\right]\Phi^*=0,\ \text{ a.e. in }(0,1),\end{aligned}\\
\\
 Z^*(0)=Z^*(1)=(Z^*)'(0)=(Z^*)'(1)=0=(Z^*)''(0)=(Z^*)''(1)=0,\\
(\Phi^*)'(0)=(\Phi^*)'(1)=(\Phi^*)'''(0)=(\Phi^*)'''(1)=0.\end{array}\right.\ \end{flalign}

\noindent\textbf{Step I.} In this step, our aim is to show that there exists $Z$ such that 
\begin{equation}\label{140}\left\{\begin{array}{l}\lambda \mathcal{L}Z+\mathcal{F}Z=0, \text{ in }(0,1),\\
Z(0)+Z(1)\neq 0 \text{ and }\left<Z,\mu_1\right>=0.\end{array}\right.\ \end{equation}
To this end, let us denote by 
$$\mathcal{H}:=\left\{V\in H^4(0,1):\ \lambda\mathcal{L}V+\mathcal{F}V=0,\ V''(0)=V''(1)=0\right\}.$$ (To recall the form of $\mathcal{L}$ and $\mathcal{F}$ see \eqref{e19}-\eqref{e20}.) We know that $\mathcal{H}$ is a linear space whose $\dim_\mathbb{C}\mathcal{H}=3.$ Next, we shall show that there cannot exist $V_1,V_2\in\mathcal{H}\cap \mathcal{S}$ which are linearly independent. But firstly, we claim that  for a nonzero $V\in\mathcal{H}\cap\mathcal{S}$ necessarily $V(0)\neq0$. Indeed, otherwise one has $\mathcal{L}V(0)=\mathcal{L}V(1)=0$. Since $V\in\mathcal{H}$, it follows that 
$$0=\lambda\mathcal{L}V+\mathcal{F}V=\lambda\mathcal{L}V-\nu(\mathcal{L}V)''+\nu(k^2+l^2)\mathcal{L}V+\text{i}kU\mathcal{L}V+\text{i}kU''V.$$Scalalry multiplying the above equation by $\mathcal{L}V$ and taking the real part of the result, it yields that 
$$\Re\lambda\|\mathcal{L}V\|^2+\nu\|(\mathcal{L}V)'\|^2+\nu(k^2+l^2)\|\mathcal{L}V\|^2=0,$$since
$$\left<U''V,\mathcal{L}V\right>=-2C_U(\|V'\|^2+(k^2+l^2)\|V\|^2)\in\mathbb{R}.$$Recalling that $\Re\lambda\geq0$, the above implies that $V\equiv 0$, which is in contradiction with our choice of $V$. 

Now, let nonzero $V_1,V_2\in \mathcal{H}\cap\mathcal{S}.$ From above we have that $V_1(0)\neq0 $ and $V_2(0)\neq0.$ Define then
$$\tilde{V}:=V_1-\frac{V_1(0)}{V_2(0)}V_2\ \in\mathcal{H}\cap\mathcal{S}.$$Since $\tilde{V}(0)=0$, it follows from above that $\tilde{V}\equiv 0$, consequently $V_1$ and $V_2$ are linearly dependent.

The same holds true for $V_1,V_2\in\mathcal{H}\cap\mathcal{AS}$. Therefore, a basis for $\mathcal{H}$ can be chosen to be of the form $\left\{V_1,V_2,V_3\right\}$, where $V_1\in\mathcal{S},\ V_2\in\mathcal{AS}$ and $V_3$ has necessarily no symmetry. Indeed, there surely exists some $V_4\in\mathcal{H}$ that has no symmetry. Then, $V_4-\hat{V_4}\in\mathcal{H}\cap\mathcal{AS}$ and $V_4+\hat{V_4}\in\mathcal{H}\cap\mathcal{S}$ are linearly independent. This set can be completed to a basis, and so the conclusion follows.

\noindent\textbf{Case a).} $V_1$ verifies $\left<\mu_1,V_1\right>=0$. From above we know that $V_1(0)\neq0$, and so, $V_1(0)+V_1(1)=2V_1(0)\neq0$. Hence, one can take $Z:=V_1$.

\noindent\textbf{Case b).} If $\left<\mu_1,V_1\right>\neq0$. Then, there are two possibilities for $V_3$: 

\textbf{Case b.1).} If $\left<\mu_1,V_3\right>=0.$ We claim that $V_3(0)+V_3(1)\neq0$. Otherwise, define
$$\tilde{V}:=V_3+\hat{V_3}.$$We have $\tilde{V}\in \mathcal{H}\cap\mathcal{S}$ and $\tilde{V}(0)=0$. From above it yields that $\tilde{V}\equiv0,$ and consequently $V_3\in\mathcal{AS}$, which is a contradiction. Hence, in this case one can take $Z:=V_3$.

\textbf{Case b.2).} If $\left<\mu_1,V_3\right>\neq0.$ We define
$$\tilde{V}:=V_1-\frac{\left<\mu_1,V_1\right>}{\left<\mu_1,V_3\right>}V_3.$$Obviously, we have $\left<\mu_1,\tilde{V}\right>=0$. Observe that $\tilde{V}(0)+\tilde{V}(1)\neq0$. Indeed, otherwise, set
$$\mathbf{V}:=\tilde{V}+\hat{\tilde{V}}.$$  We have $\mathbf{V}\in\mathcal{H}\cap\mathcal{S}$ and $\mathbf{V}(0)=0$. From above we know that this implies $\mathbf{V}\equiv0.$ Hence, $\tilde{V}\in\mathcal{H}\cap\mathcal{AS}$. Clearly seen, $\tilde{V}\not\equiv0$, and so, from above we know that necessarily $\tilde{V}$ and $V_2$ are linearly dependent. This means that there exists $\chi\in\mathbb{C}$ such that $\tilde{V}=\chi V_2$, or, equivalently
$$ V_1-\frac{\left<\mu_1,V_1\right>}{\left<\mu_1,V_3\right>}V_3-\chi V_2=0,$$which contradicts the fact that $V_1,V_2,V_3$ form a basis. In conclusion, $\tilde{V}(0)+\tilde{V}(1)\neq0$ and since $\left<\mu_1,\tilde{V}\right>=0$, in this case one can take $Z:=\tilde{V}.$

The next two steps concern the function $\Phi^*$.

\noindent\textbf{Step II.} In this step, we want to show that there exists $\Phi$ such that 
\begin{equation}\label{e150}\left\{\begin{array}{l}\mathcal{E}\Phi+\lambda\Phi=0,\text{ in }(0,1),\\
\Phi'(0)=\Phi'(1)=0 \text{ and }\Phi'''(0)-\Phi'''(1)\neq0,\\
\left<\mu_2,\Phi\right>=0.\end{array}\right.\ \end{equation}(To recall the form of $\mathcal{E}$ see \eqref{e21}.)

The proof is similar with the one in the previous step. That is why, we shall only sketch it. Let us denote by
$$\mathcal{K}:=\left\{ V\in H^4(0,1): \mathcal{E}V+\lambda V=0,\ V'(0)=V'(1)=0 \right\},$$which has dimension equal to three. We claim that there cannot exist nonzero $V_1,V_2\in\mathcal{K}\cap\mathcal{AS}$ which are linearly independent. To see this, we firstly observe that $V_1'''(0)\neq0$ and $V_2'''(0)\neq0$. Indeed, otherwise, one has
$$\mathcal{E}V_1+\lambda V_1=0 \text{ and }V_1'(0)=V_1'(1)=V_1'''(0)=V_1'''(1)=0.$$Scalarly multiplying the above equation by $V_1$ and taking the real part of the result, we get that
$$\rho_0\varepsilon\|V_1''\|^2+[2\rho_0\varepsilon(k^2+l^2)+\gamma]\|V_1'\|^2+[\rho_0\varepsilon(k^2+l^2)^2+\gamma(k^2+l^2)+\Re\lambda]\|V_1\|^2=0.$$Recalling the $\Re\lambda\geq0$, from the above we deduce that $V_1\equiv0$, which is a contradiction. To show the linear dependence, we define
$$\tilde{V}:=V_1-\frac{V_1'''(0)}{V_2'''(0)}V_2,$$and since $\tilde{V}\in\mathcal{K}\cap\mathcal{AS}$ and $\tilde{V}'''(0)=0$, the conclusion follows immediately. The same holds true for $V_1,V_2\in\mathcal{K}\cap\mathcal{S}$. Therefore, as before a basis in $\mathcal{K}$ can be taken of the form $\left\{V_1,V_2,V_3\right\}$ with $V_1\in\mathcal{S}$, $V_2\in\mathcal{AS}$ and $V_3$ with no symmetry.

\noindent\textbf{Case a).} If $\left<\mu_2,V_2\right>=0.$  From above we know that $V_2'''(0)-V_2'''(1)=2V_2'''(0)\neq0$. Thus, one can take $\Phi:=V_2$.

\noindent\textbf{Case b).} If $\left<\mu_2,V_2\right>\neq0$ and

\textbf{Case b.1).} $\left<\mu_2, V_3\right>=0.$ Then, in this case one can show that $V_3'''(0)-V_3'''(1)\neq0$, because otherwise one can define 
$$\tilde{V}:=V_3+\hat{V}_3$$ for which $\tilde{V}\in\mathcal{K}\cap\mathcal{S}$ and $\tilde{V}'''(0)=0$. This yields that $V_3$ should be antisymmetric, which is false. In this case one can take $\Phi:=V_3$.

\textbf{Case b.2).} If $\left<\mu_2,V_3\right>\neq0,$ then defining 
$$\tilde{V}:=V_2-\frac{\left<\mu_2,V_2\right>}{\left<\mu_2,V_3\right>}V_3$$we get that $\left<\mu_2,\tilde{V}\right>=0$ and $\tilde{V}'''(0)-\tilde{V}'''(1)\neq0$, since otherwise $\tilde{V}$ should be antisymmetric. Note that if $\tilde{V}$ is  antisymmetric, then from above we know that necessarily $\tilde{V}$ and $V_2$ are linearly dependent. So, there exists $\chi\in\mathbb{C}$ such that $\tilde{V}=\chi V_2$, or, equivalently
$$(1-\chi)V_2= \frac{\left<\mu_2,V_2\right>}{\left<\mu_2,V_3\right>}V_3.$$But this implies that $V_3\in\mathcal{AS},$ which is absurd.  Hence, in this case one can take $\Phi=\tilde{V}$.

\noindent\textbf{Step III.} With similar arguments as above, one may show that there exists $\Psi$ such that
\begin{equation}\label{e154}\left\{\begin{array}{l}\mathcal{E}\Psi+\lambda\Psi=0,\\
\Psi'''(0)=\Psi'''(1)=0 \text{ and }\Psi'(0)-\Psi'(1)\neq0,\\
\left<\mu_2,\Psi\right>=0.\end{array}\right.\ \end{equation}

\noindent\textbf{Step IV.} In this step, once we have in hand the functions $Z,\Phi,\Psi$ satisfying \eqref{140}, \eqref{e150} and \eqref{e154}, respectively, we scalarly multiply the first equation in \eqref{e137} by $Z$ and the second equation of \eqref{e137} successively by $\Phi$ and $\Psi$, respectively. It yields that

$$\begin{aligned}& \nu (Z^*)'''(y)\overline{Z(y)}|_0^1+\left<Z^*,\lambda\mathcal{L}Z+\mathcal{F}Z\right>+\left<\mu_1,Z\right>=0,\\&
\left<\mu_2,\Phi\right>+\rho_0\varepsilon\Phi^*(y)\overline{\Phi'''(y)}|_0^1+\left<\Phi^*,\lambda\Phi+\mathcal{E}\Phi\right>=0,\\&
\left<\mu_2,\Psi\right>+\rho_0\varepsilon(\Phi^*)''(y)\overline{\Psi'(y)}|_0^1+\left<\Phi^*,\lambda\Psi+\mathcal{E}\Psi\right>=0.\end{aligned}$$Which in virtue of the symmetry of $Z^*$ and $\Phi^*$, and the fact that $Z,\Phi,\Psi$ satisfy the systems \eqref{140}, \eqref{e150} and \eqref{e154}, respectively, we get that
$$(Z^*)'''(0)[\overline{Z(0)+Z(1)}]=0,\ \Phi^*(0)[\overline{\Phi'''(0)-\Phi'''(1)}]=0 \text{ and }(\Phi^*)''(0)[\overline{\Psi'(0)-\Psi'(1)}]=0,$$ and so
 $(Z^*)'''(0)=\Phi^*(0)=(\Phi^*)''(0)=0.$ 

Now, let us recap, $Z^*$ and $\Phi^*$ satisfy the following fourth order system
\begin{flalign}\label{e155}\left\{\begin{array}{l}\begin{aligned}&  \nu (Z^*)^{iv}-(2 \nu(k^2+l^2)-\text{i}kU+\overline{\lambda})(Z^*)''-2\text{i}kU'(Z^*)'\\&
+[(k^2+l^2)\overline{\lambda}+\nu (k^2+l^2)^2-\text{i}k(k^2+l^2)U]Z^*+\varphi'_\infty\Phi^*=0, \end{aligned}\\
\\
\begin{aligned}&-\varepsilon \kappa(k^2+l^2)L_{kl}(\varphi'_\infty Z^*)-\varepsilon\kappa(k^2+l^2)\varphi'''_\infty Z^*
+\rho_0\varepsilon(\Phi^*)^{iv}-\left[2\rho_0\varepsilon(k^2+l^2)+\gamma\right](\Phi^*)''\\&+\left[\rho_0\varepsilon(k^2+l^2)^2+\gamma(k^2+l^2)+\overline{\lambda}-\text{i}kU\right]\Phi^*=0,\ \text{ a.e. in }(0,1),\end{aligned}\\
\\
 Z^*(0)=(Z^*)'(0)=(Z^*)''(0)=(Z^*)'''(0)=0,\\
(\Phi^*)(0)=(\Phi^*)'(0)=(\Phi^*)''(0)=(\Phi^*)'''(0)=0.\end{array}\right.\ \end{flalign}
It is obvious that, due to the fact that $(Z^*,\Phi^*)$ is analytic, the null boundary conditions in \eqref{e155} imply that $(Z^*,\Phi^*)\equiv(0,0),$ which is absurd.

To conclude with the proof, we recall that we showed that if $(Z^*,\Phi^*)$ is an eigenvector corresponding to an unstable eigenvalue, then necessarily either $(Z^*)''(0)\neq0$ or $(Z^*)''(1)\neq0$. Since there are only a finite number of unstable eigenvectors, it yields the existence of some $a\in\mathbb{C}$ such that, each unstable eigenvector satisfies
$$  (Z^*)''(0)+a(Z^*)''(1)\neq0.$$ The proof is complete.

\end{proof}

For the sake of  the simplicity of the presentation we shall strengthen assumption $(H_1)$ by
$$(H_{11})\ \ \text{All the unstable eigenvlues $\lambda_{j}^{kl},\ k^2+l^2\leq M,\ j=1,2,..., N_{kl},$ are simple.}$$
The present algorithm works equally-well in the case of semisimple eigenvalues (see for details
\cite{book}). But, we shall not develop this subject here since the presentation may get too hard to
follow.

 By using (if necessary) the Gram-Schmidt procedure, we may assume that the systems $\left\{\left(Z_j^{kl}, \Phi_j^{kl}\right)\right\}_{j=1}^{N_{kl}}$ and  $\left\{\left(Z_j^{kl*}, \Phi_j^{kl*}\right)\right\}_{j=1}^{N_{kl}}$ are bi-orthonormal, i.e.,
\begin{flalign}\label{e73}\left<\left(Z_i^{kl}, \Phi_i^{kl}\right),\left(Z_j^{kl*}, \Phi_j^{kl*}\right)\right>=\delta_{ij},\end{flalign}where $\delta_{ij}$ stands for the Kronecker delta function. (The preocedure is as follows: we fix the set $\left\lbrace \left(Z_j^{kl*}, \Phi_j^{kl*}\right)\right\rbrace $) guaranteed by Lemma \ref{l2}, then via the Gram-Schmidt procedure we construct the bi-orthogonal set $\left\lbrace \left(Z_i^{kl}, \Phi_i^{kl}\right)\right\rbrace $ .)

At this point, we are able to introduce  the stabilizing feedback  $\mathbf{s}_{kl},\mathbf{t}_{kl}\ 0<k^2+l^2\leq M$. More exactly:
\begin{flalign}\label{e60}\mathbf{s}_{kl}(t):=-\frac{i}{k}\Omega_{kl}(t) \text{ and }\mathbf{t}_{kl}(t)=\overline{a}\frac{\text{i}}{k}\Omega_{kl}(t),\end{flalign}
where $a\in\mathbb{C}$ is given by Lemma \ref{l2}, and  $\Omega_{kl}=\Omega_{kl}(t,v,\varphi)$ has the form
\begin{flalign}\label{e61}\begin{aligned}&\Omega_{kl}(t):=\\&
\left<\Lambda_{sum}^{kl}\mathbf{R}_{kl}\left(\begin{array}{c}\int_\mathcal{O}\left[-v_{yy}+(k^2+l^2)v\right]\overline{Z_1^{kl*}}+\varphi\overline{\Phi_1^{kl*}}]e^{-\text{i}kx}e^{-\text{i}lz}dxdydz\\
\int_\mathcal{O}\left[-v_{yy}+(k^2+l^2)v\right]\overline{Z_2^{kl*}}+\varphi\overline{\Phi_2^{kl*}}]e^{-\text{i}kx}e^{-\text{i}lz}dxdydz\\
\ddots\\
\int_\mathcal{O}\left[-v_{yy}+(k^2+l^2)v\right]\overline{Z_{N_{kl}}^{kl*}}+\varphi\overline{\Phi_{N_{kl}}^{kl*}}]e^{-\text{i}kx}e^{-\text{i}lz}dxdydz\end{array}\right),\left(\begin{array}{c}(Z_{1}^{kl*})''(0)+a(Z_1^{kl*})''(1)\\ (Z_{2}^{kl*})''(0)+a(Z_2^{kl*})''(1)\\ \ddots \\ (Z_{N_{kl}}^{kl*})''(0)+a(Z_{N_{kl}}^{kl*})''(1)\end{array}\right)\right>_{N_{kl}}\\&
=\left<\Lambda_{sum}^{kl}\mathbf{R}_{kl}\left(\begin{array}{c}\left<(\mathcal{L}_{kl}v_{kl},\varphi_{kl}),(Z_1^{kl*},\Phi_1^{kl*})\right>\\\left<(\mathcal{L}_{kl}v_{kl},\varphi_{kl}),(Z_2^{kl*},\Phi_2^{kl*})\right>\\
\ddots\\
\left<(\mathcal{L}_{kl}v_{kl},\varphi_{kl}),(Z_{N_{kl}}^{kl*},\Phi_{N_{kl}}^{kl*})\right>\end{array}\right),\left(\begin{array}{c}(Z_{1}^{kl*})''(0)+a(Z_1^{kl*})''(1)\\ (Z_{2}^{kl*})''(0)+a(Z_2^{kl*})''(1)\\ \ddots \\ (Z_{N_{kl}}^{kl*})''(0)+a(Z_{N_{kl}}^{kl*})''(1)\end{array}\right)\right>_{N_{kl}}\end{aligned},&&\end{flalign}with $\Lambda_{sum}^{kl}:=\overline{\Lambda_{\gamma_1^{kl}}^{kl}+...+\Lambda_{\gamma^{kl}_{N_{kl}}}^{kl}},$ for $\Lambda_{\gamma_i^{kl}}^{kl}$ the following $N_{kl}$ diagonal matrices
\begin{flalign}\label{e63}\Lambda_{\gamma_i^{kl}}^{kl}:=diag\left(\frac{1}{\gamma_i^{kl}+\lambda_1^{kl}},\frac{1}{\gamma_i^{kl}+\lambda_2^{kl}}, ... ,\frac{1}{\gamma_i^{kl}+\lambda_{N_{kl}}^{kl}}  \right),\ i=1,2,...,N_{kl},\end{flalign}for some $0<\gamma_1^{kl}<...<\gamma_{N_{kl}}^{kl}$, $N_{kl}$ real constants, sufficiently large such as relation \eqref{e76} below holds true. Moreover,
 \begin{flalign}\label{e66}\mathbf{R}_{kl}:=(R_1^{kl}+R_2^{kl}+...+R_{N_{kl}}^{kl})^{-1},\end{flalign}where
\begin{flalign}\label{e67}R_i^{kl}:=\overline{\Lambda_{\gamma_i^{kl}}^{kl}}R_{kl}\Lambda_{\gamma_i^{kl}}^{kl},\ i=1,2,...,N_{kl}.\end{flalign} Here, $R_{kl}$ stands for the square matrix of order $N_{kl}$ of the form
\begin{equation}\label{e200}R_{kl}:=\left(\begin{array}{cccc}l_1^{kl}\overline{l_1^{kl}}& l_1^{kl}\overline{l_2^{kl}}& ...& l_1^{kl}\overline{l_{N_{kl}}^{kl}}\\
l_2^{kl}\overline{l_1^{kl}}& l_2^{kl}\overline{l_2^{kl}}& ...& l_2^{kl}\overline{l_{N_{kl}}^{kl}}\\
\ddots& \ddots &\ddots&\ddots\\
l_{N_{kl}}^{kl}\overline{l_1^{kl}}& l_{N_{kl}}^{kl}\overline{l_2^{kl}}& ...& l_{N_{kl}}^{kl}\overline{l_{N_{kl}}^{kl}}\end{array}\right),\end{equation}where
$$l_j^{kl}:=(Z_{j}^{kl*})''(0)+a(Z_j^{kl*})''(1),\ j=1,2,...,N_{kl}.$$ Recall that, by Lemma \ref{l2}, $l_j^{kl}\neq0$, for all $j$. Hence, with similar arguments as in \cite[Lemma 5.2]{ion2}, one may show that the sum $R_1^{kl}+R_2^{kl}+...+R_{N_{kl}}^{kl}$ is indeed an invertible matrix.  We recall that $\left<\cdot,\cdot\right>_{N}$ stands for the classical euclidean scalar product in $\mathbb{C}^N$.

We plug controllers $\mathbf{s}_{kl},\mathbf{t}_{kl}$, given by \eqref{e60}, into system \eqref{e12} and show that they achieve the exponential asymptotic stability of it. So, our aim  now is to show the asymptotic exponential stability of the following closed-loop system
\begin{flalign}\label{e79}\left\{\begin{array}{l}(\mathcal{L}_{kl}v_{kl})_t+\mathcal{F}_{kl}v_{kl}-\varepsilon\kappa(k^2+l^2)\phi_\infty'\mathcal{L}_{kl}\varphi_{kl}-\varepsilon\kappa(k^2+l^2)\varphi'''_\infty\varphi_{kl}=0,\\
(\varphi_{kl})_t+\phi'_\infty v_{kl}+\mathcal{E}_{kl}\varphi_{kl}=0,\ \text{a.e. in } (0,1),\\
 v_{kl}(0)=v_{kl}(1)=0,\\
v'_{kl}(0)=-\Omega_{kl}(t,v_{kl},\varphi_{kl}),v'_{kl}(1)=\overline{a}\Omega_{kl}(t,v_{kl},\varphi_{kl}),\\
\varphi'_{kl}(0)=\varphi'_{kl}(1)=\varphi'''_{kl}(0)=\varphi'''_{kl}(1)=0,
\end{array}\right.\ &&\end{flalign}$0<k^2+l^2\leq M$, where $\Omega_{kl}$ is described in \eqref{e61}. To this end, we equivalently rewrite the feedback $\Omega_{kl}$ as the sum 
$$\Omega_{kl}= \psi^{kl}_1+\psi^{kl}_2+...+\psi^{kl}_{N_{kl}},$$ where
\begin{flalign}\label{1hum}\psi^{kl}_i(t):=\left<\mathbf{R}_{kl}\left(\begin{array}{c}\left<(\mathcal{L}_{kl}v_{kl}(t),\varphi_{kl}),(Z^{kl*}_1,\Phi_1^{kl*})\right>\smallskip\\
\left<(\mathcal{L}_{kl}v_{kl}(t),\varphi_{kl}),(Z^{kl*}_2,\Phi_2^{kl*})\right>\\\multicolumn{1}{c}{.......................}\\ \left<(\mathcal{L}_{kl}v_{kl}(t),\varphi_{kl}),(Z^{kl*}_{N_{kl}},\Phi_{N_{kl}}^{kl*})\right>\end{array}\right),\left(\begin{array}{c}\frac{1}{\gamma^{kl}_i+\lambda^{kl}_1}\left[(Z_{1}^{kl*})''(0)+a(Z_1^{kl*})''(1)\right]\smallskip\\ \frac{1}{\gamma^{kl}_i+\lambda^{kl}_2}\left[(Z_{2}^{kl*})''(0)+a(Z_2^{kl*})''(1)\right]\\ \multicolumn{1}{c}{.......................} \\\frac{1}{\gamma^{kl}_i+\lambda^{kl}_{N_{kl}}}\left[(Z_{N_{kl}}^{kl*})''(0)+a(Z_{N_{kl}}^{kl*})''(1)\right]\end{array}\right)\right>_{N_{kl}}
,\ t\geq0,
&&\end{flalign} for $i=1,2,...,N_{kl}$. 

We continue to follow the ideas in \cite{book}. More exactly, we shall lift the boundary control, $\Omega_{kl}$, into the equations, via the lifting operators $\mathbb{D}_{\gamma_i^{kl}}(\psi):=(V,\Phi)$, where $(V,\Phi)$ is solution to
\begin{flalign}\label{e71}\left\{\begin{array}{l}\mathcal{F}_{kl}V-\varepsilon\kappa(k^2+l^2)\varphi'_\infty\mathcal{L}_{kl}\Phi-\varepsilon\kappa(k^2+l^2)\varphi'''_\infty\Phi+\displaystyle 2\sum_{j=1}^{N_{kl}}\lambda_j^{kl}\left<\mathcal{L}_{kl}V,Z_j^{kl*}\right>Z_j^{kl}+\gamma_i^{kl}\mathcal{L}_{kl}V=0\\
\varphi'_\infty V+\mathcal{E}_{kl}\Phi+\displaystyle 2\sum_{j=1}^{N_{kl}}\lambda_j^{kl}\left<\Phi,\Phi_j^{kl*}\right>\Phi_j^{kl}+\gamma_i^{kl}\Phi=0,\text{ a.e. in }(0,1),\\
V(0)= V(1)=0, V'(0)=-\psi, V'(1)=\overline{a}\psi,\\
\Phi'(0)=\Phi'(1)=\Phi'''(0)=\Phi'''(1)=0, \ 1\leq i\leq N_{kl}.\end{array}\right.\ &&\end{flalign} Here, $\psi\in\mathbb{C}$ and  \begin{flalign}\label{e76}\begin{aligned}& 0<\gamma_1^{kl}<\gamma_2^{kl}<...<\gamma_{N_{kl}}^{kl} \textit{ are chosen sufficiently large such that }\\&
\textit{ equation \eqref{e71} is well-posed, with }(V,\Phi)\in H^\frac{3}{2}(0,1)\times H^\frac{3}{2}(0,1).\end{aligned}\end{flalign}
We shall denote by 
$$\mathbb{D}^1_{\gamma_i^{kl}}(\psi):=V \text{ and by }\mathbb{D}^2_{\gamma_i^{kl}}(\psi):=\Phi.$$For latter purpose, let us compute the scalar product 
$$\left<(\mathcal{L}_{kl}\mathbb{D}^1_{\gamma_i^{kl}}(\psi),\mathbb{D}^2_{\gamma_i^{kl}}(\psi)),(Z_j^{kl*},\Phi_j^{kl*})\right>,$$ for all $1\leq i,j\leq N_{kl}$. We have by scalarly multiplying system \eqref{e71} by $(Z_j^{kl*},\Phi_j^{kl*})$ and using the biorthogonality \eqref{e73},  that
$$\begin{aligned}  & 0=\psi\overline{(Z_j^{kl*})''(0)+a(Z_{j}^{kl*})''(1)}+\left<\mathcal{L}_{kl}\mathbb{D}^1_{\gamma_i^{kl}}(\psi),\mathcal{L}_{kl}^{-1}\mathcal{F}_{kl}Z_j^{kl*}\right>+\left<\mathbb{D}^2_{\gamma_i^{kl}}(\psi),-\varepsilon\kappa(k^2+l^2)\mathcal{L}_{kl}(\varphi'_\infty Z_j^{kl*})\right>\\&
+2\lambda_j^{kl}\left<\mathcal{L}_{kl}\mathbb{D}^1_{\gamma_i^{kl}}(\psi),Z_j^{kl*}\right>+\gamma^{kl}_i\left<\mathcal{L}_{kl}\mathbb{D}^1_{\gamma_i^{kl}}(\psi),Z_j^{kl*}\right>+\left<\mathcal{L}_{kl}\mathbb{D}^1_{\gamma_i^{kl}}(\psi),\mathcal{L}_{kl}^{-1}(\varphi'_\infty\Phi_j^{kl*})\right>\\&
-\left<\mathcal{L}_{kl}\mathbb{D}^1_{\gamma_i^{kl}}(\psi),\varepsilon\kappa(k^2+l^2)\varphi'''_\infty\Phi_j^{kl*}\right>\\&
+\left<\mathbb{D}^2_{\gamma_i^{kl}}(\psi),\mathcal{E}^*_{kl}\Phi_j^{kl*}\right>+2\lambda_j^{kl}\left<\mathbb{D}^2_{\gamma_i^{kl}}(\psi),\Phi_j^{kl*}\right>+\gamma_i^{kl}\left<\mathbb{D}^2_{\gamma_i^{kl}}(\psi),\Phi_j^{kl*}\right>\\&
=\psi\overline{(Z_j^{kl*})''(0)+a(Z_{j}^{kl*})''(1)} +(2\lambda_j^{kl}+\gamma_i^{kl})\left<(\mathcal{L}_{kl}\mathbb{D}^1_{\gamma_i^{kl}}(\psi),\mathbb{D}^2_{\gamma_i^{kl}}(\psi)),(Z_j^{kl*},\Phi_j^{kl*})\right>\\&
+\left<(\mathcal{L}_{kl}\mathbb{D}^1_{\gamma_i^{kl}}(\psi),\mathbb{D}^2_{\gamma_i^{kl}}(\psi)),\mathbf{A}^*_{kl}(Z_j^{kl*},\Phi_j^{kl*})\right>.\end{aligned}$$Recalling that $\mathbf{A}_{kl}^*(Z_j^{kl*},\Phi_j^{kl*})=-\overline{\lambda_j^{kl}}(Z_j^{kl*},\Phi_j^{kl*})$, in virtue of \eqref{e40} and \eqref{e74}, we deduce from above that
\begin{flalign}\label{e75}\left<(\mathcal{L}_{kl}\mathbb{D}^1_{\gamma_i^{kl}}(\psi),\mathbb{D}^2_{\gamma_i^{kl}}(\psi)),(Z_j^{kl*},\Phi_j^{kl*})\right>=-\frac{\psi}{\lambda_j^{kl}+\gamma_i^{kl}}\overline{Z_{j}^{kl*})''(0)+a(Z_j^{kl*})''(1)}\neq0.,\ 1\leq i,j\leq N_{kl}.\end{flalign}

Now, recalling the definition of $\psi_i^{kl}$ in \eqref{1hum} and using \eqref{e75}, we deduce as in \cite[Eq. (4.4)]{ion2} that
\begin{flalign}\label{8hum}\begin{aligned}&\left(\begin{array}{c}\left<(\mathcal{L}_{kl}\mathbb{D}^1_{\gamma^{kl}_i}(\psi_i^{kl}),\mathbb{D}^2_{\gamma^{kl}_i}(\psi_i^{kl})),(Z_1^{kl*},\Phi_1^{kl*})\right>\\ \left<(\mathcal{L}_{kl}\mathbb{D}^1_{\gamma^{kl}_i}(\psi_i^{kl}),\mathbb{D}^2_{\gamma^{kl}_i}(\psi_i^{kl})),(Z_2^{kl*},\Phi_2^{kl*})\right>\\\multicolumn{1}{c}{.......................}\\\left<(\mathcal{L}_{kl}\mathbb{D}^1_{\gamma^{kl}_i}(\psi_i^{kl}),\mathbb{D}^2_{\gamma^{kl}_i}(\psi_i^{kl})),(Z_{N_{kl}}^{kl*},\Phi_{N_{kl}}^{kl*})\right>\end{array}\right)
=-R_i^{kl}\mathbf{R}_{kl}\left(\begin{array}{c}\left<(\mathcal{L}_{kl}v_{kl}(t),\varphi_{kl}),(Z^{kl*}_1,\Phi_1^{kl*})\right>\smallskip\\
\left<(\mathcal{L}_{kl}v_{kl}(t),\varphi_{kl}),(Z^{kl*}_2,\Phi_2^{kl*})\right>\\\multicolumn{1}{c}{.......................}\\ \left<(\mathcal{L}_{kl}v_{kl}(t),\varphi_{kl}),(Z^{kl*}_{N_{kl}},\Phi_{N_{kl}}^{kl})\right>\end{array}\right),\end{aligned}&&\end{flalign}  for $i=1,...,N_{kl}$. 

In equation \eqref{e79}, we denote by 
$$\left(\begin{array}{c}Z \\ \Phi \end{array}\right):=\left(\begin{array}{c}L_{kl}\left[v_{kl}-\displaystyle \sum_{i=1}^{N_{kl}}\mathbb{D}^1_{\gamma^{kl}_i}(\psi_i^{kl})\right] \\ \varphi_{kl}- \displaystyle \sum_{i=1}^{N_{kl}}\mathbb{D}^2_{\gamma^{kl}_i}(\psi_i^{kl})\end{array}\right).$$ Then, subtracting \eqref{e71}  and \eqref{e79}, we deduce that
\begin{flalign}\label{e12hum}
\begin{aligned}
\left(\begin{array}{c}Z \\ \Phi\end{array}\right)_t&=\displaystyle-\mathbf{A}_{kl}\left(\begin{array}{c}Z \\ \Phi\end{array}\right)  +2\sum_{i,j=1}^{N_{kl}}\lambda^{kl}_j\left<\left(\begin{array}{c}\mathcal{L}_{kl}\mathbb{D}^1_{\gamma^{kl}_i}(\psi^{kl}_i) \\ \mathbb{D}^2_{\gamma^{kl}_i}(\psi^{kl}_i)\end{array}\right),\left(\begin{array}{c}Z_j^{kl*} \\ \Phi_j^{kl*}\end{array}\right)\right>\left(\begin{array}{c}Z_j^{kl} \\ \Phi_j^{kl}\end{array}\right)\\&
+\sum_{i=1}^{N_{kl}}\gamma^{kl}_i\left(\begin{array}{c}\mathcal{L}_{kl}\mathbb{D}^1_{\gamma^{kl}_i}(\psi^{kl}_i) \\ \mathbb{D}^2_{\gamma^{kl}_i}(\psi^{kl}_i)\end{array}\right)-\sum_{i=1}^{N_{kl}}\left(\begin{array}{c}\mathcal{L}_{kl}\mathbb{D}^1_{\gamma^{kl}_i}(\psi^{kl}_i) \\ \mathbb{D}^2_{\gamma^{kl}_i}(\psi^{kl}_i)\end{array}\right)_t.\end{aligned}&&\end{flalign}
The perturbation of $-\mathbf{A}_{kl}$ can be written in terms of $(Z,\Phi)$, since, similarly as in \cite[Eq. (4.7)]{ion2},  one can show that the feedback laws $\psi_i^{kl},\ i=1,..,N_{kl},$ can be equivalently rewritten as
\begin{flalign}\label{e80}\psi^{kl}_i(t):=\frac{1}{2}\left<\mathbf{R}_{kl}\left(\begin{array}{c}\left<(Z(t),\Phi(t)),(Z^{kl*}_1,\Phi_1^{kl*})\right>\smallskip\\
\left<(Z(t),\Phi(t)),(Z^{kl*}_2,\Phi_2^{kl*})\right>\\\multicolumn{1}{c}{.......................}\\ \left<(Z(t),\Phi(t)),(Z^{kl*}_{N_{kl}},\Phi_{N_{kl}}^{kl*})\right>\end{array}\right),\left(\begin{array}{c}\frac{1}{\gamma^{kl}_i+\lambda^{kl}_1}\left[(Z_{1}^{kl*})''(0)+a(Z_1^{kl*})''(1)\right]\smallskip\\ \frac{1}{\gamma^{kl}_i+\lambda^{kl}_2}\left[(Z_{2}^{kl*})''(0)+a(Z_2^{kl*})''(1)\right]\\ \multicolumn{1}{c}{.......................} \\\frac{1}{\gamma^{kl}_i+\lambda^{kl}_{N_{kl}}}\left[(Z_{N_{kl}}^{kl*})''(0)+a(Z_{N_{kl}}^{kl*})''(1)\right]\end{array}\right)\right>_{N_{kl}}
,\ t\geq0.
&&\end{flalign}  Moreover, likewise in \cite[Eq. (4.8)]{ion2}, we have now
\begin{flalign}\label{e81}\begin{aligned}&\left(\begin{array}{c}\left<(\mathcal{L}_{kl}\mathbb{D}^1_{\gamma^{kl}_i}(\psi_i^{kl}),\mathbb{D}^2_{\gamma^{kl}_i}(\psi_i^{kl})),(Z_1^{kl*},\Phi_1^{kl*})\right>\\ \left<(\mathcal{L}_{kl}\mathbb{D}^1_{\gamma^{kl}_i}(\psi_i^{kl}),\mathbb{D}^2_{\gamma^{kl}_i}(\psi_i^{kl})),(Z_2^{kl*},\Phi_2^{kl*})\right>\\\multicolumn{1}{c}{.......................}\\\left<(\mathcal{L}_{kl}\mathbb{D}^1_{\gamma^{kl}_i}(\psi_i^{kl}),\mathbb{D}^2_{\gamma^{kl}_i}(\psi_i^{kl})),(Z_{N_{kl}}^{kl*},\Phi_{N_{kl}}^{kl*})\right>\end{array}\right)
=-\frac{1}{2}R_i^{kl}\mathbf{R}_{kl}\left(\begin{array}{c}\left<(Z(t),\Phi(t)),(Z^{kl*}_1,\Phi_1^{kl*})\right>\smallskip\\
\left<(Z(t),\Phi(t)),(Z^{kl*}_2,\Phi_2^{kl*})\right>\\\multicolumn{1}{c}{.......................}\\ \left<(Z(t),\Phi(t)),(Z^{kl*}_{N_{kl}},\Phi_{N_{kl}}^{kl})\right>\end{array}\right),\end{aligned}&&\end{flalign}for $i=1,...,N_{kl}$.

Next, we decompose system  (\ref{e12hum}) into its stable and unstable part. To this end, we introduce  the projections $P_N$, and its adjoint $P_N^*$, defined by
 $$P_N:=\frac{1}{2\pi \text{i}} \int_{\Gamma}{(\lambda I+\mathbf{A}_{kl})^{-1}d\lambda};\  P_N^*:=\frac{1}{2\pi\text{i}} \int_{\bar{\Gamma}}{(\lambda I+\mathbf{A}_{kl}^*)^{-1}d\lambda},$$
 where $\Gamma$ (its conjugate $\bar{\Gamma}$, respectively) separates the unstable spectrum from the stable one of
 $-\mathbf{A}_k$
 ($-\mathbf{A}_k^*$, respectively). We set
 \begin{flalign}\label{f12} -\mathbf{A}_{N}^u:=P_{N}(-\mathbf{A}_{kl}),\ -\mathbf{A}_{N}^s:=(I-P_{N})(-\mathbf{A}_{kl}),
 \end{flalign}
 for the restrictions of $-\mathbf{A}_{kl}$,
 respectively. Then, the system (\ref{e12hum}) can accordingly be decomposed as
$$\begin{aligned} \left(\begin{array}{c}Z\\ \Phi\end{array}\right)&=P_N\left(\begin{array}{c}Z\\ \Phi\end{array}\right)+(I-P_N)\left(\begin{array}{c}Z\\ \Phi\end{array}\right)\\&
=:\left(\begin{array}{c} Z_N^u\\ \Phi_N^u\end{array}\right)+\left(\begin{array}{c}Z_N^s\\ \Phi_N^s\end{array}\right)\end{aligned}$$
  where applying $P_{N}$ and $(I-P_{N})$ on (\ref{e12hum}), we obtain
 \begin{flalign}\label{f100} \begin{aligned}  &\frac{d}{dt}\left(\begin{array}{c} Z_N^u\\ \Phi_N^u\end{array}\right)+\mathbf{A}^u_{N}\left(\begin{array}{c} Z_N^u \\ \Phi_N^u\end{array}\right)
=P_{N}\left\{2\sum_{i,j=1}^{N_{kl}}\lambda^{kl}_j\left<\left(\begin{array}{c}\mathcal{L}_{kl}\mathbb{D}^1_{\gamma^{kl}_i}(\psi^{kl}_i) \\ \mathbb{D}^2_{\gamma^{kl}_i}(\psi^{kl}_i)\end{array}\right),\left(\begin{array}{c}Z_j^{kl*} \\ \Phi_j^{kl*}\end{array}\right)\right>\left(\begin{array}{c}Z_j^{kl} \\ \Phi_j^{kl}\end{array}\right)\right.\ \\& \left.\ 
+\sum_{i=1}^{N_{kl}}\gamma^{kl}_i\left(\begin{array}{c}\mathcal{L}_{kl}\mathbb{D}^1_{\gamma^{kl}_i}(\psi^{kl}_i) \\ \mathbb{D}^2_{\gamma^{kl}_i}(\psi^{kl}_i)\end{array}\right)-\sum_{i=1}^{N_{kl}}\left(\begin{array}{c}\mathcal{L}_{kl}\mathbb{D}^1_{\gamma^{kl}_i}(\psi^{kl}_i) \\ \mathbb{D}^2_{\gamma^{kl}_i}(\psi^{kl}_i)\end{array}\right)_t \right\}\end{aligned}&&\end{flalign}
and
\begin{flalign}\label{f101}\begin{aligned} &\frac{d}{dt}\left(\begin{array}{c} Z_N^s\\ \Phi_N^s\end{array}\right)+\mathbf{A}^s_{N}\left(\begin{array}{c} Z_N^s \\ \Phi_N^s\end{array}\right)
=(I-P_{N})\left\{2\sum_{i,j=1}^{N_{kl}}\lambda^{kl}_j\left<\left(\begin{array}{c}\mathcal{L}_{kl}\mathbb{D}^1_{\gamma^{kl}_i}(\psi^{kl}_i) \\ \mathbb{D}^2_{\gamma^{kl}_i}(\psi^{kl}_i)\end{array}\right),\left(\begin{array}{c}Z_j^{kl*} \\ \Phi_j^{kl*}\end{array}\right)\right>\left(\begin{array}{c}Z_j^{kl} \\ \Phi_j^{kl}\end{array}\right)\right.\ \\& \left.\ 
+\sum_{i=1}^{N_{kl}}\gamma^{kl}_i\left(\begin{array}{c}\mathcal{L}_{kl}\mathbb{D}^1_{\gamma^{kl}_i}(\psi^{kl}_i) \\ \mathbb{D}^2_{\gamma^{kl}_i}(\psi^{kl}_i)\end{array}\right)-\sum_{i=1}^{N_{kl}}\left(\begin{array}{c}\mathcal{L}_{kl}\mathbb{D}^1_{\gamma^{kl}_i}(\psi^{kl}_i) \\ \mathbb{D}^2_{\gamma^{kl}_i}(\psi^{kl}_i)\end{array}\right)_t \right\}
 \end{aligned}&&\end{flalign}
 respectively.

 Let us decompose $\left(\begin{array}{c} Z_N^u, \Phi_N^u\end{array}\right)$ as
 $$\left(\begin{array}{c} Z_N^u\\ \Phi_N^u\end{array}\right)(t)=\sum_{j=1}^{N_{kl}}\left<\left(\begin{array}{c} Z_N^u\\ \Phi_N^u\end{array}\right)(t),\left(\begin{array}{c}Z_j^{kl*}\\ \Phi_j^{kl*}\end{array}\right)\right> \left(\begin{array}{c}Z_j^{kl}\\ \Phi_j^{kl}\end{array}\right).$$
  We insert this $\left(\begin{array}{c} Z_N^u\\ \Phi_N^u\end{array}\right)$ in equation (\ref{f100}). Then, we scalarly multiply \eqref{f100} successively by $\left(\begin{array}{c}Z_j^{kl*}\\ \Phi_j^{kl*}\end{array}\right),$ $j=1,...,N_{kl}$, take account of the bi-orthogonality of the eigenfunctions systems, notice that we may assume that 
	$$P_{N}^*\left(\begin{array}{c}Z_j^{kl*}\\ \Phi_j^{kl*}\end{array}\right)=\left(\begin{array}{c}Z_j^{kl*}\\ \Phi_j^{kl*}\end{array}\right),$$ since $P^*_{N}$ is idempotent; and take advantage of relation (\ref{e81}), to get that
	\begin{flalign}\label{e44ee}\mathcal{Z}_t=-\gamma^{kl}_1\mathcal{Z}+\sum_{i=2}^{N_{kl}}(\gamma^{kl}_1-\gamma^{kl}_i)R^{kl}_i\mathbf{R}_{kl}\mathcal{Z},\ t\geq0,\end{flalign}
	where $\mathcal{Z}$ is the vector containing the first $N_{kl}$ modes, i.e.
	$$\mathcal{Z}:=\left(\begin{array}{c}\left<\left(\begin{array}{c} Z_N^u, \Phi_N^u\end{array}\right)(t),\left(\begin{array}{c}Z_1^{kl*}, \Phi_1^{kl*}\end{array}\right)\right>\\ \left<\left(\begin{array}{c} Z_N^u, \Phi_N^u\end{array}\right)(t),\left(\begin{array}{c}Z_2^{kl*}, \Phi_2^{kl*}\end{array}\right)\right> \\ \multicolumn{1}{c}{.................}\\ \left<\left(\begin{array}{c} Z_N^u, \Phi_N^u\end{array}\right)(t),\left(\begin{array}{c}Z_{N_{kl}}^{kl*}, \Phi_{N_{kl}}^{kl*}\end{array}\right)\right>\end{array}\right)$$ and $\Lambda^k:=\text{diag }(\lambda^{kl}_1,\lambda^{kl}_2,...,\lambda^{kl}_{N_{kl}}).$  Scalarly multiplying in $\mathbb{C}^{N_{kl}}$ equation \eqref{e44ee} by $\mathbf{R}_{kl}\mathcal{Z}$, and taking into account that $R_i^{kl}$ are positive definite, we deduce that
	$$\|\mathcal{Z}(t)\|_{N_{kl}}\leq Ce^{-\eta t}\|\mathcal{Z}(0)\|_{N_{kl}}, \forall t\geq0,$$ for some $C,\eta$ independent of $k$ and $l$. (For more details see \cite[Eqs. (4.9)-(4.15)]{ion2}.) Here, $\| \cdot\|_{N_{kl}}$ stands for the euclidean norm in $\mathbb{C}^{N_{kl}}$. The above relation says that the unstable part of the system (\ref{e12hum}) is, in fact, stable. Furthermore,  one  can  argue likewise in the end of the proof of \cite[ Theorem 2.1]{ion2}, in order to deduce  that the system (\ref{f101}) is stable as-well, and to conclude that
	once the feedback $\mathbf{s}_{kl}=\mathbf{t}_{kl}=0$  for $k^2+l^2>M$, and $\mathbf{s}_{kl}=-\frac{\text{i}}{k}\Omega_{kl}$ and $\mathbf{t}_{kl}=\overline{a}\frac{\text{i}}{k}\Omega_{kl}$, for $k^2+l^2\leq M,\ k,l\in\mathbb{Z}\setminus\left\{0\right\}$ is inserted into the system \eqref{e11}, it yields that the corresponding solution of the closed loop equation \eqref{e11}  satisfies for some $C_1,\eta_1>0$, independent of $k$ or $l$
	\begin{equation}\label{e110}\|(u_{kl}(t),v_{kl}(t),w_{kl}(t),\varphi_{kl}(t))\|^2\leq C_1e^{-\eta_1 t}\|(u_{kl}(0),v_{kl}(0),w_{kl}(0),\varphi_{kl}(0))\|^2,\ \forall t\geq0,\end{equation} for all $k,l\in\mathbb{Z}\setminus\left\{0\right\}.$ The details are omitted.

\noindent \textbf{2. The case $k=l=0$.} Setting $k=l=0$ and $\mathbf{s}_{00}\equiv\mathbf{t}_{00}\equiv 0$ in \eqref{e11}, we get the following system
\begin{flalign}\label{e90}\left\{\begin{array}{l}
(u_{00})_t-\nu u''_{00}+U'v_{00}=0,\\
(v_{00})_t-\nu v''_{00}+p'_{00}=-\varepsilon\kappa\Delta\varphi_\infty\varphi'_{00}-\varepsilon\kappa\varphi'_\infty\varphi''_{00},\\
(w_{00})_t-\nu w''_{00}=0,\\
v'_{00}=0,\\
(\varphi_{00})_t+\varphi'_\infty v_{00}+\rho_0\varepsilon\varphi^{iv}_{00}-\gamma\varphi''_{00}=0, \text{ a.e. in }(0,1),\\
u_{00}(0)=u_{00}(1)=v_{00}(0)=v_{00}(1)=w_{00}(0)=w_{00}(1)=0,\\
\varphi'_{00}(0)=\varphi'_{00}(1)=\varphi'''_{00}(0)=\varphi'''_{00}(1)=0.\end{array} \right.\ &&\end{flalign} 
By the fourth equation in \eqref{e90} and the null boundary conditions for $v_{00}$ we immediately get that $v_{00}\equiv0$. Using this in the first and the third equation of \eqref{e90}, after a convenient scalar multiplication and using the Poincare inequality, we deduce, for some $C>0$, that
$$\frac{1}{2}\frac{d}{dt}\left(\|u_{00}(t)\|^2+\|w_{00}(t)\|^2\right)+C \left(\|u_{00}(t)\|^2+\|w_{00}(t)\|^2\right)\leq 0, \forall t\geq0.$$Hence,
\begin{flalign}\label{e91}\|u_{00}(t)\|^2+\|w_{00}(t)\|^2\leq C_2e^{-\eta_2 t}\left(\|u_{00}(0)\|^2+\|w_{00}(0)\|^2\right), \forall t\geq0, \end{flalign} for some $C_2,\eta_2>0$.

Next, the fifth equation in \eqref{e90} reads as
$$(\varphi_{00})_t+\rho_0\varepsilon\varphi^{iv}_{00}-\gamma\varphi''_{00}=0,$$
$$\varphi'(0)=\varphi'(1)=\varphi'''(0)=\varphi'''(1).$$ Set 
$$\mathbf{A}_{00}\varphi=\rho_0\varepsilon\varphi^{iv}-\gamma\varphi'',$$ for all 
$$\varphi \in \mathcal{D}(\mathbf{A}_{00})=\left\{\varphi\in H^4(0,1):\ \varphi'(0)=\varphi'(1)=\varphi'''(0)=\varphi'''(1)\right\}.$$ Similarly as in Lemma \ref{l1}, one can show that $-\mathbf{A}_{00}$ generates a $C_0-$semigroup in $H$, has compact resolvent, and consequently it has a countable set of eigenvalues, which are semi-simple since $\mathbf{A}_{00}$ is self-adjoint. Let 
$$\lambda\varphi+\mathbf{A}_{00}\varphi=0.$$Scalarly multiplying the above equation by $\varphi$ it yields that
$$\lambda \|\varphi\|^2+\rho_0\varepsilon\|\varphi''\|^2+\gamma\|\varphi'\|^2=0,$$hence, necessarily $\lambda<0$. In other words, the operator $-\mathbf{A}_{00}$ is stable. Therefore, there exists some $C_3,\eta_3>0$ such that
\begin{flalign}\label{e93}\|\varphi_{00}(t)\|^2\leq C_3e^{-\eta_3t}\|\varphi_{00}(0)\|^2,\ \forall t\geq0.\end{flalign}
The conclusion of this section is that \eqref{e91} together with \eqref{e93} imply the existence of some $C_4,\eta_4>0$ such that
\begin{flalign}\label{e94}\|(u_{00}(t),v_{00}(t),w_{00}(t),\varphi_{00}(t))\|^2\leq C_4e^{-\eta_4 t}\|(u_{00}(0),v_{00}(0),w_{00}(0),\varphi_{00}(0))\|^2,\ \forall t\geq0.\end{flalign}

\noindent\textbf{3. The case $k=0$ and $l\neq0$.} System \eqref{e11}, with null boundary control $\mathbf{s}_{0l}\equiv\mathbf{t}_{0l}\equiv0,$ reads as
\begin{flalign}\label{e95}\left\{\begin{array}{l}(u_{0l})_t-\nu u''_{0l}+\nu l^2u_{0l}+U'v_{0l}=0,\\
(v_{0l})_t-\nu v''_{0l}+\nu l^2v_{0l}+p'_{0l}=-\varepsilon\kappa\Delta\varphi_\infty\varphi'_{0l}-\varepsilon\kappa\varphi'_\infty(-l^2\varphi_{0l}+\varphi''_{0l}),\\
(w_{0l})_t-\nu w''_{0l}+\nu l^2w_{0l}+\text{i}lp_{0l}=-\text{i}l\varepsilon\kappa\Delta\varphi_\infty\varphi_{0l},\\
v'_{0l}+\text{i}lw_{0l}=0,\\
(\varphi_{0l})_t+\varphi'_\infty v_{0l}+\rho_0\varepsilon\varphi^{iv}_{0l}-(2\rho_0\varepsilon l^2+\gamma)\varphi''_{0l}+(\rho_0\varepsilon l^4+\gamma l^2)\varphi_{0l}=0, \\
\text{ a.e. in }(0,1),\\
u_{0l}(0)=u_{0l}(1)=v_{0l}(0)=v_{0l}(1)=w_{0l}(0)=w_{0l}(1)=0,\\
\varphi_{0l}'(0)=\varphi'_{0l}(1)=\varphi'''_{0l}(0)=\varphi'''_{0l}(1)=0. \end{array}\right.\ &&\end{flalign}
In a similar manner as we did in the case $k,l\neq0$, we reduce the pressure from the second and the third equation of system \eqref{e95}, and use the fourth relation of it. Then, we couple the result with the fifth equation in system \eqref{e95}. We get that
\begin{flalign}\label{e96}\left\{\begin{array}{l}(-v''_{0l}+l^2v_{0l})_t+\nu v^{iv}_{0l}-2\nu l^2v''_{0l}+\nu l^4v_{0l}=-\varepsilon\kappa\varphi'_\infty l^2(-l^2\varphi_{0l}+\varphi''_{0l})+l^2\varepsilon\kappa\varphi'''_\infty\varphi_{0l},\\
(\varphi_{0l})_t+\varphi'_\infty v_{0l}+\rho_0\varepsilon\varphi^{iv}_{0l}-(2\rho_0\varepsilon l^2+\gamma)\varphi''_{0l}+(\rho_0\varepsilon l^4+\gamma l^2)\varphi_{0l}=0, \\
\text{ a.e. in }(0,1),\\
v_{0l}(0)=v_{0l}(1)=v'_{0l}(0)=v'_{0l}(1)=0,\\
\varphi_{0l}'(0)=\varphi'_{0l}(1)=\varphi'''_{0l}(0)=\varphi'''_{0l}(1)=0.
\end{array}\right.\ &&\end{flalign}   
Denote by $\Phi:=-\varphi''_{0l}+l^2\varphi_{0l}$. Then, the above system can be equivalently written as
\begin{flalign}\label{e97}\left\{\begin{array}{l}(-v''_{0l}+l^2v_{0l})_t+\nu v^{iv}_{0l}-2\nu l^2v''_{0l}+\nu l^4v_{0l}-\varepsilon\kappa\varphi'_\infty l^2\Phi-l^2\varepsilon\kappa\varphi'''_\infty\varphi_{0l}=0,\\
(\overline{\varphi_{0l}})_t+\varphi'_\infty \overline{v_{0l}}-\rho_0\varepsilon \overline{\Phi''}+(\rho_0\varepsilon l^2+\gamma)\overline{\Phi}=0,\text{ a.e. in }(0,1),\\
v_{0l}(0)=v_{0l}(1)=v'_{0l}(0)=v'_{0l}(1)=0,\\
\Phi'(0)=\Phi'(1)=0.
\end{array}\right.\ &&\end{flalign}If we test by $\overline{v_{0l}}$ the first equation in \eqref{e97} and add the result to the second equation in \eqref{e97} tested by $\varepsilon\kappa l^2\Phi$, then take the real part of the result, we arrive at
$$\begin{aligned} &\frac{1}{2}\frac{d}{dt}[\|v'_{0l}\|^2+l^2\|v_{0l}\|^2+\varepsilon\kappa l^2(\|\varphi'_{0l}\|^2+l^2\|\varphi_{0l}\|^2)]+\nu\|v''_{0l}\|^2+2\nu l^2\|v'_{0l}\|^2\\&+\nu l^4 \|v_{0l}\|^2+\rho_0\varepsilon \|\Phi'\|^2+(\rho^2_0+\gamma)\|\Phi\|^2= l^2\varepsilon\kappa\Re\int_0^1\varphi'''_\infty\varphi_{0l}\overline{v_{0l}}dy.\end{aligned}$$ Easily seen, via a proper use of the Young's inequality (as we did several times before), the above implies the existence of some $C_5,\eta_5>0$, independent of $l$, such that
\begin{flalign}\label{e99}\|v_{0l}(t)\|^2+\|v'_{0l}(t)\|^2+\|\varphi_{0l}(t)\|^2\leq C_5 e^{-\eta_5 t}(\|v_{0l}(0)\|^2+\|\varphi_{0l}(0)\|^2),\ \forall t\geq0.&&\end{flalign}Next, we return to the first equation in \eqref{e95}, scalarly multiply it by $u_{0l}$, use the Poincare inequality and relation \eqref{e99}, to get that
\begin{flalign}\label{e100}\|u_{0l}(t)\|^2\leq C_6e^{-\eta_6 t}\|u_{0l}(0)\|^2,\ \forall t\geq0,\end{flalign}
for some $C_6,\eta_6>0$, independent of $l$. Finally, the fourth equation in \eqref{e95} combined with relation \eqref{e99}, then with relation \eqref{e100}, it yields the existence of some $C_7,\eta_7>0$, independent of $l$, such that
\begin{flalign}\label{e101}\|(u_{0l}(t),v_{0l}(t),w_{0l}(t),\varphi_{0l}(t))\|^2\leq C_7e^{-\eta_7 t}\|(u_{0l}(0),v_{0l}(0),w_{0l}(0),\varphi_{0l}(0))\|^2,\ \forall t\geq0.\end{flalign}

\noindent \textbf{4.The case $k\neq 0$ and $l=0$.} This case reduces to the first case as we shall  see at once. System \eqref{e11} reads as
\begin{flalign}\label{e102}\left\{\begin{array}{l}(u_{k0})_t- \nu(-k^2u_{k0}+u''_{k0})+\text{i}kUu_{k0}+U'v_{k0}+\text{i}kp_{k0}=-\text{i}k\varepsilon\kappa\Delta\varphi_\infty\varphi_{k0},\\
(v_{k0})_t- \nu(-l^2v_{k0}+v''_{k0})+\text{i}kUv_{k0}+p'_{k0}=-\varepsilon\kappa\Delta\varphi_\infty\varphi'_{k0}-\varepsilon\kappa\varphi'_\infty(-k^2\varphi_{k0}+\varphi''_{k0}),\\
(w_{k0})_t- \nu(-l^2w_{k0}+w''_{k0})+\text{i}kUw_{k0}=0,\\
\text{i}ku_{k0}+v'_{k0}=0,\\
(\varphi_{k0})_t+\varphi'_\infty v_{k0}+\rho_0\varepsilon\varphi_{k0}^{iv}-\left(\rho_0\varepsilon 2k^2+\gamma\right)\varphi''_{k0}+\left(\rho_0\varepsilon k^4+\gamma k^2+\text{i}kU\right]\varphi_{k0}=0,\\
\text{ a.e. in }(0,1),\\
u_{k0}(0)=\mathbf{s}_{k0},\ u_{k0}(1)=\mathbf{t}_{k0},\ v_{k0}(0)=v_{k0}(1)=w_{k0}(0)=w_{k0}(1)=0,\\
\varphi'_{k0}(0)=\varphi'_{k0}(1)=\varphi'''_{k0}(0)=\varphi'''_{k0}(1)=0.\end{array}\right.\ &&\end{flalign}Let us start with $w_{k0}$. Scalalrly multiplying the third equation in \eqref{e102} by $w_{k0}$, taking the real part of the result, and using the Poincare inequality, we immediately deduce that
\begin{flalign}\label{e103}\|w_{k0}(t)\|^2\leq C_8e^{-\eta_8t}\|w_{k0}(0)\|^2,\ \forall t\geq0,\end{flalign}for some $C_8,\eta_8>0$, independent of $k$. 

Now, let us reduce the pressure from the first two equations in \eqref{e102}, use the fourth equation, then couple the result with the fifth equation in \eqref{e102}. We obtain that
\begin{flalign}\label{ce12}\left\{\begin{array}{l}\begin{aligned} &[-v''_{k0} +k^2v_{k0}]_t + \nu v^{iv}_{k0}-(2 \nu k^2+\text{i}kU)v''_{k0}+ (\nu k^4+\text{i}k^3U+\text{i}kU'')v_{k0}\\&
+k^2\varepsilon\kappa\varphi'_\infty(-k^2\varphi_{k0}+\varphi''_{k0})+k^2\varepsilon\kappa\varphi'''_\infty\varphi_{k0}=0,\end{aligned}\\
(\varphi_{k0})_t+\varphi'_\infty v_{k0}+\rho_0\varepsilon\varphi_{k0}^{iv}-\left(\rho_02\varepsilon k^2+\gamma\right)\varphi''_{k0}+\left(\rho_0\varepsilon k^4+\gamma k^2+\text{i}kU\right)\varphi_{k0}=0,\\
\text{ a.e. in }(0,1),\\
v_{k0}(0)=v_{k0}(1)=0, \ v'_{k0}(0)=-\text{i}k\mathbf{s}_{k0},v'_{k0}(1)=-\text{i}k\mathbf{t}_{k0},\\
\varphi'_{k0}(0)=\varphi'_{k0}(1)=\varphi'''_{k0}(0)=\varphi'''_{k0}(1)=0.\end{array}\right.\ &&\end{flalign}One can see that this system is similar with \eqref{e12}. Thus, we introduce the operators: $L_{k0}:\mathcal{D}(L_{k0})\subset H\rightarrow H,$ $F_{k0}:\mathcal{D}(F_{k0})\subset H\rightarrow H $ and $E_{k0}:\mathcal{D}(E_{kl})\subset H\rightarrow H$, as follows
\begin{flalign}\label{ce13}\begin{aligned}& L_{k0}v:=-v''+k^2v,
\forall v\in \mathcal{D}(L_{k0})=H_0^1(0,1)\cap H^2(0,1);\end{aligned}\end{flalign}
\begin{flalign}\label{ce14}\begin{aligned}& F_{k0}v:= \nu v^{iv}-(2\nu k^2+\text{i}kU)v''+(\nu k^4+\text{i}k^3U+\text{i}kU'')v,\\&
\forall v\in \mathcal{D}(F_{k0})=H_0^2(0,1)\cap H^4(0,1);\end{aligned}\end{flalign}and
\begin{flalign}\label{ce15}\begin{aligned}& E_{k0}\varphi:=\rho_0\varepsilon\varphi^{iv}-\left(2\rho_0\varepsilon k^2+\gamma\right)\varphi''+\left(\rho_0\varepsilon k^4+\gamma k^2+\text{i}kU\right)\varphi,\\&
\forall \varphi\in \mathcal{D}(E_{k0})=\left\{\varphi\in H^4(0,1):\ \varphi'(0)=\varphi'(1)=\varphi'''(0)=\varphi'''(1)=0\right\}.\end{aligned}\end{flalign} Then, we set $\mathbb{A}_{k0}:\mathcal{D}(\mathbb{A}_{k0})\subset H\rightarrow H,$ as
\begin{flalign}\label{ce26}\mathbb{A}_{k0}:=F_{k0}L^{-1}_{k0},\ \mathcal{D}(\mathbb{A}_{k0})=\left\{z\in H:\ L^{-1}_{k0}z\in\mathcal{D}(F_{k0})\right\}.\end{flalign} And finally, define $\mathbf{A}_{k0}:\mathcal{D}(\mathbf{A}_{k0})\subset H\times H\rightarrow H\times H,$ as
\begin{flalign}\label{ce25}\mathbf{A}_{k0}\left(\begin{array}{c}z \\ \varphi\end{array}\right):=\left(\begin{array}{cc}\mathbb{A}_{k0} & -\varepsilon\kappa k^2\varphi'_\infty L_{k0}-k^2\varepsilon\kappa\varphi'''_\infty\\
\varphi'_\infty L_{k0}^{-1} & E_{k0}\end{array}\right)\left(\begin{array}{c}z \\  \varphi\end{array}\right),\end{flalign}for all
$$\left(\begin{array}{c}z \\ \varphi\end{array}\right)\in \mathcal{D}(\mathbf{A}_{k0})=\mathcal{D}(\mathbb{A}_{k0})\times \mathcal{D}(E_{k0}).$$ Of course, one may show that regarding the operators $\mathbf{A}_{k0}$, we have the counterpart of the two Lemmas \ref{l1}, \ref{l2}. More precisely,
\begin{lemma}\label{cl1}The operator $-\mathbf{A}_{k0}$ generates a $C_0-$analytic semigroup on $\mathcal{D}(L_{k0}^{-1})\times H$, and for each $\lambda\in\rho(-\mathbf{A}_{k0})$, $(\lambda I+\mathbf{A}_{k0})^{-1}$ is compact. Moreover, one has for each $\eta>0$ there exists $M>0$, sufficiently large such that
\begin{flalign}\label{ce30}\begin{aligned}&\sigma(-\mathbf{A}_{k0})\subset\left\{\lambda\in\mathbb{C}:\ \Re\lambda\leq -\eta\right\},\forall |k|> M.\end{aligned}\end{flalign}
\end{lemma}Similarly as in \eqref{e59}, one may show that for $|k|>M$, we have 
\begin{flalign}\label{ce59}\begin{aligned}&\|u_{k0}(t)\|^2+\|v_{k0}(t)\|^2+\|w_{k0}(t)\|^2+\|\varphi_{k0}(t)\|^2\\&\leq C_9e^{-\eta_9t}(\|u_{k0}(0)\|^2+\|v_{k0}(0)\|^2+\|w_{k0}(0)\|^2+\|\varphi_{k0}(0)\|^2),\ \forall |k|>M,&&\end{aligned}\end{flalign}for some $C_9,\eta_9>0$, independent of $k$.
Therefore, we have to control the system \eqref{e102} for $|k|\leq M$ only.

Furthermore, Lemma \ref{cl1} guarantees that, for each  $k$, $-\mathbf{A}_{k0}$ has a countable set of eigenvalues, denoted by $\left\{\lambda_j^{k0}\right\}_{j=1}^\infty$. Moreover,  there is only a finite number $N_{k0}\in\mathbb{N}$ of eigenvalues with $\Re \lambda_j^{k0}\geq0,$  the unstable eigenvalues. Let us denote by $\left\{\left(Z_j^{k0}, \Phi_j^{k0}\right)\right\}_{j=1}^\infty$ and by $\left\{\left(Z_j^{k0*}, \Phi_j^{k0*}\right)\right\}_{j=1}^\infty$ the corresponding eigenvectors systems of $-\mathbf{A}_{k0}$ and its dual $-\mathbf{A}^*_{k0},$ respectively.

As before, we assume that
$$(H_{22})\text{ All the unstable eigenvalues $\lambda_{j}^{k0},\ |k|\leq M,\ j=1,2,..., N_{k0},$ are semisimple.}$$
Then, we have the counterpart of Lemma \ref{l2}
\begin{lemma}\label{cl2} Under assumption $(H_{22})$, there exists a $b\in\mathbb{C}$, such that for each unstable eigenvalue $\overline{\lambda_j^{k0}}$ ($|k|\leq M$ with $M$ from Lemma \ref{cl1}; and $j\in\left\{1,2,...,N_{k0}\right\}),$  one can choose the corresponding eigenvector $\left(Z_j^{k0*}, \Phi_j^{k0*}\right)$   in such a way that
\begin{flalign}\label{ce40}(Z_{j}^{k0*})''(0)+b(Z_{j}^{k0*})''(1)\neq0.\end{flalign}
\end{lemma}

 Finally, similarly as in \eqref{e60}, we define the stabilizing feedback laws, for $|k|\leq M,\ k\neq0$:
\begin{flalign}\label{ce60}\mathbf{s}_{k0}(t):=-\frac{i}{k}\Omega_{k0}(t) \text{ and }\mathbf{t}_{k0}:=\overline{b}\frac{i}{k}\Omega_{k0},\end{flalign}
where  $\Omega_{k0}=\Omega_{k0}(t,v,\varphi)$ has the form
\begin{flalign}\label{ce61}\begin{aligned}&\Omega_{k0}(t):=\\&
\left<\Lambda_{sum}^{k0}\mathbf{R}_{k0}\left(\begin{array}{c}\int_\mathcal{O}\left[-v_{yy}+k^2v\right]\overline{Z_1^{k0*}}+\varphi\overline{\Phi_1^{k0*}}]e^{-\text{i}kx}dxdydz\\
\int_\mathcal{O}\left[-v_{yy}+k^2v\right]\overline{Z_2^{k0*}}+\varphi\overline{\Phi_2^{k0*}}]e^{-\text{i}kx}dxdydz\\
\ddots\\
\int_\mathcal{O}\left[-v_{yy}+k^2v\right]\overline{Z_{N_{k0}}^{k0*}}+\varphi\overline{\Phi_{N_{k0}}^{k0*}}]e^{-\text{i}kx}dxdydz\end{array}\right),\left(\begin{array}{c}(Z_{1}^{k0*})''(0)+b(Z_1^{k0*})''(1)\\ (Z_{2}^{k0*})''(0)+b(Z_2^{k0*})''(1)\\ \ddots \\ (Z_{N_{k0}}^{k0*})''(0)+b(Z_{N_{k0}}^{k0*})''(1)\end{array}\right)\right>_{N_{k0}},
\end{aligned}&&\end{flalign}with $\Lambda_{sum}^{k0}:=\overline{\Lambda_{\gamma_1^{k0}}^{k0}+...+\Lambda_{\gamma^{k0}_{N_{k0}}}^{k0}},$ for $\Lambda_{\gamma_i^{k0}}^{k0}$ the following $N_{k0}$ diagonal matrices
\begin{flalign}\label{ce63}\Lambda_{\gamma_i^{k0}}^{k0}:=diag\left(\frac{1}{\gamma_i^{k0}+\lambda_1^{k0}},\frac{1}{\gamma_i^{k0}+\lambda_2^{k0}}, ... ,\frac{1}{\gamma_i^{k0}+\lambda_{N_{k0}}^{k0}}  \right),\ i=1,2,...,N_{k0},\end{flalign}for some sufficiently large $0<\gamma_1^{k0}<...<\gamma_{N_{k0}}^{k0}$. Moreover,
 \begin{flalign}\label{ce66}\mathbf{R}_{k0}:=(R_1^{k0}+R_2^{k0}+...+R_{N_{k0}}^{k0})^{-1},\end{flalign}where
\begin{flalign}\label{ce67}R_i^{k0}:=\overline{\Lambda_{\gamma_i^{k0}}^{k0}}R_{k0}\Lambda_{\gamma_i^{k0}}^{k0},\ i=1,2,...,N_{k0}.\end{flalign} Here, $R_{k0}$ stands for the square matrix of order $N_{k0}$, of the form
\begin{equation}\label{ce200}R_{k0}:=\left(\begin{array}{cccc}l_1^{k0}\overline{l_1^{k0}}& l_1^{k0}\overline{l_2^{k0}}& ...& l_1^{k0}\overline{l_{N_{k0}}^{k0}}\\
l_2^{k0}\overline{l_1^{k0}}& l_2^{k0}\overline{l_2^{k0}}& ...& l_2^{k0}\overline{l_{N_{k0}}^{k0}}\\
\ddots& \ddots &\ddots&\ddots\\
l_{N_{k0}}^{k0}\overline{l_1^{k0}}& l_{N_{k0}}^{k0}\overline{l_2^{k0}}& ...& l_{N_{k0}}^{k0}\overline{l_{N_{k0}}^{k0}}\end{array}\right),\end{equation}where
$$l_j^{k0}:=(Z_{j}^{k0*})''(0)+b(Z_j^{k0*})''(1),\ j=1,2,...,N_{k0}.$$ 

Similarly as in the first case, $k,l\neq0$, one may show that once we plug controllers $\mathbf{s}_{k0}=\mathbf{t}_{k0}=0$ for $|k|>M$ and $\mathbf{s}_{k0},\mathbf{t}_{k0}$ given by \eqref{ce60}, for $|k|\leq M$, into system \eqref{ce12}, the corresponding solution of the closed-loop system satisfies for some $C_{10},\eta_{10}>0,$ independent of $k$ and $l$,
 \begin{flalign}\label{ce101}\|(u_{k0}(t),v_{k0}(t),w_{k0}(t),\varphi_{k0}(t))\|^2\leq C_{10}e^{-\eta_{10} t}\|(u_{k0}(0),v_{k0}(0),w_{k0}(0),\varphi_{k0}(0))\|^2,\ \forall t\geq0.\end{flalign}
 
Gathering together relations \eqref{e59},\eqref{e110},\eqref{e94},\eqref{e101},\eqref{e103},\eqref{ce101}, we see that the feedback laws defined by \eqref{e60} and \eqref{ce60}, respectively, ensure the stability of the linearized system, thereby completing the proof. $\hfill\Box$

\section*{Acknowledgement} This  was supported by a grant of the Romanian Ministry of Research
and Innovation, CNCS--UEFISCDI, project number
PN-III-P1-1.1-TE-2019-0397, within PNCDI III.

\end{document}